\documentclass[10pt]{amsart}
\usepackage{amssymb} 
\usepackage{amsthm,amsmath}
\usepackage{epic,cite}

\title[Quasitrivial semigroups: characterizations and enumerations]{Quasitrivial semigroups: characterizations and enumerations}

\author{Miguel Couceiro}
\address{LORIA, (CNRS - Inria Nancy Grand Est - Universit\'e de Lorraine), BP239, 54506 Vandoeuvre-l\`es-Nancy, France} \email{miguel.couceiro[at]\{inria,loria\}.fr }

\author{Jimmy Devillet}
\address{Mathematics Research Unit, University of Luxembourg, Maison du Nombre, 6, avenue de la Fonte, L-4364 Esch-sur-Alzette, Luxembourg}
\email{jimmy.devillet[at]uni.lu}

\author{Jean-Luc Marichal}\thanks{Corresponding author: Jean-Luc Marichal is with the Mathematics Research Unit, University of Luxembourg, Maison du Nombre, 6, avenue de la Fonte, L-4364 Esch-sur-Alzette, Luxembourg.\\ Email: jean-luc.marichal[at]uni.lu}
\address{Mathematics Research Unit, University of Luxembourg, Maison du Nombre, 6, avenue de la Fonte, L-4364 Esch-sur-Alzette, Luxembourg}
\email{jean-luc.marichal[at]uni.lu}

\date{January 29, 2018}

\theoremstyle{plain}
\newtheorem{theorem}{Theorem}[section]
\newtheorem{lemma}[theorem]{Lemma}
\newtheorem{proposition}[theorem]{Proposition}
\newtheorem{corollary}[theorem]{Corollary}
\newtheorem{fact}[theorem]{Fact}

\theoremstyle{definition}
\newtheorem{definition}[theorem]{Definition}
\newtheorem{example}[theorem]{Example}

\theoremstyle{remark}

\newtheorem{remark}{Remark}

\newcommand{\R}{\mathbb{R}}

\newcommand{\cleq}{\mathrm{conv}_{\leq}}

\begin{document}
\begin{abstract}
We investigate the class of quasitrivial semigroups and provide various characterizations of the subclass of quasitrivial and commutative semigroups as well as the subclass of quasitrivial and order-preserving semigroups. We also determine explicitly the sizes of these classes when the semigroups are defined on finite sets. As a byproduct of these enumerations, we obtain several new integer sequences.
\end{abstract}



\maketitle

\section{Introduction}

Let $X$ be an arbitrary nonempty set. We use the symbol $X_n$ if $X$ contains $n\geq 1$ elements, in which case we assume without loss of generality that $X_n=\{1,\ldots,n\}$.

In this paper we investigate the class of binary operations $F\colon X^2\to X$ that are associative and quasitrivial, where quasitriviality means that $F$ always outputs one of its input values. In the algebraic language, the pair $(X,F)$ is then called a \emph{quasitrivial semigroup} (for general background, see, e.g., \cite{Jez78,Kep81,Lan80} and for a recent reference, see \cite{Ack}). We also investigate certain subclasses of quasitrivial semigroups by adding properties such as commutativity, order-preservation, and the existence of neutral elements. The case where the semigroups are defined on finite sets (i.e., $X=X_n$ for any integer $n\geq 1$) is of particular interest as it enables us to address and solve various enumeration issues. We remark that most of our results rely on a simple known theorem (Theorem~\ref{thm:kimura}) that provides a descriptive characterization of the class of quasitrivial semigroups.

After presenting some definitions and preliminary results (including Theorem~\ref{thm:kimura}) in Section 2, we provide in Section 3 different characterizations of the class of quasitrivial and commutative (i.e., Abelian) semigroups on both arbitrary sets and finite sets (Theorem~\ref{thm:char1}). In the latter case we illustrate some of our results by showing the contour plots of the operations. When $X$ is endowed with a total ordering we also characterize the subclass of quasitrivial, commutative, and order-preserving semigroups (Theorem~\ref{thm:cchar1}) by means of the single-peakedness property, which is a generalization to arbitrary totally ordered sets of a notion introduced 70 years ago in social choice theory. In Section 4 we introduce the ``weak single-peakedness'' property (Definition~\ref{de:wBlack}) as a further generalization of single-peakedness to arbitrary weakly ordered sets to characterize the class of quasitrivial and order-preserving semigroups (Theorem~\ref{thm:QtAsNdKi}). In the special case where the semigroups are defined on finite sets, the class of quasitrivial semigroups is also finite. This raises the problem of computing the size of this class as well as the sizes of all subclasses discussed in this paper. We tackle this problem in Section 4 where we arrive at some known integer sequences as well as new ones. The number of quasitrivial semigroups on $X_n$ for any integer $n\geq 1$ (Theorem~\ref{thm:Nk}) gives rise to a sequence that was previously unknown in the Sloane's On-Line Encyclopedia of Integer Sequences (OEIS, see \cite{Slo}). All the (old and new) sequences that we consider are given in explicit forms (i.e., closed-form expressions) and/or through their generating functions or exponential generating functions (Theorem~\ref{thm:Nk} and Propositions~\ref{prop:qe}, \ref{prop:vn}, \ref{prop:ven}, \ref{prop:rn}, and \ref{prop:ren}). In Section 5 we further investigate the single-peakedness and weak single-peakedness properties and provide a graphical characterization of weakly single-peaked weak orderings (Theorem~\ref{thm:VLL5}). We also observe that the weakly single-peaked weak orderings on finite sets are precisely the so-called single-plateaued weak orderings introduced in social choice theory.

\section{Preliminaries}

Recall that a binary relation $R$ on $X$ is said to be
\begin{itemize}
\item \emph{total} if $\forall x,y$: $xRy$ or $yRx$;
\item \emph{transitive} if $\forall x,y,z$: $xRy$ and $yRz$ implies $xRz$;
\item \emph{antisymmetric} if $\forall x,y$: $xRy$ and $yRx$ implies $x=y$.
\end{itemize}
Note that any total binary relation $R$ on $X$ is \emph{reflexive}, i.e., $xRx$ for all $x\in X$.

Recall also that a \emph{total ordering on $X$} is a binary relation $\leq$ on $X$ that is total, transitive, and antisymmetric. More generally, a \emph{weak ordering on $X$} is a binary relation $\lesssim$ on $X$ that is total and transitive. We denote the symmetric and asymmetric parts of $\lesssim$ by $\sim$ and $<$, respectively. Thus, $x\sim y$ means that $x\lesssim y$ and $y\lesssim x$. Also, $x<y$ means that $x\lesssim y$ and $\neg(y\lesssim x)$. Recall also that $\sim$ is an equivalence relation on $X$ and that $<$ induces a total ordering on the quotient set ${X/\sim}$. Thus, defining a weak ordering on $X$ amounts to defining an ordered partition of $X$. For any $a\in X$, we use the notation $[a]_{\sim}$ to denote the equivalence class of $a$, i.e., $[a]_{\sim}=\{x\in X: x\sim a\}$.

For any total ordering $\leq$ on $X$, the pair $(X,\leq)$ is called a \emph{totally ordered set} or a \emph{chain}. Similarly, for any weak ordering $\lesssim$ on $X$, the pair $(X,\lesssim)$ is called a \emph{weakly ordered set}. For any integer $n\geq 1$, we assume without loss of generality that the pair $(X_n,\leq_n)$ represents the set $X_n=\{1,\ldots,n\}$ endowed with the total ordering relation $\leq_n$ defined by $1<_n\cdots <_nn$.

If $(X,\lesssim)$ is a weakly ordered set, an element $a\in X$ is said to be \emph{maximal} (resp.\ \emph{minimal}) \emph{for $\lesssim$} if $x\lesssim a$ (resp.\ $a\lesssim x$) for all $x\in X$. We denote the set of maximal (resp.\ minimal) elements of $X$ for $\lesssim$ by $\max_{\lesssim}X$ (resp.\ $\min_{\lesssim}X$). Note that this set need not be nonempty (consider, e.g., the set of nonnegative integers endowed with the usual total ordering $\leq$).

An operation $F\colon X^2\to X$ is said to be
\begin{itemize}
\item \emph{associative} if $F(F(x,y),z)=F(x,F(y,z))$ for all $x,y,z\in X$;
\item \emph{idempotent} if $F(x,x)=x$ for all $x\in X$;
\item \emph{quasitrivial} (or \emph{conservative}) if $F(x,y)\in\{x,y\}$ for all $x,y\in X$;
\item \emph{commutative} if $F(x,y)=F(y,x)$ for all $x,y\in X$;
\item \emph{$\leq$-preserving} for some total ordering $\leq$ on $X$ if for any $x,y,x',y'\in X$ such that $x\leq x'$ and $y\leq y'$, we have $F(x,y)\leq F(x',y')$.
\end{itemize}

Given a weak ordering $\lesssim$ on $X$, the \emph{maximum} (resp.\ \emph{minimum}) \emph{operation on $X$ for $\lesssim$} is the commutative binary operation $\max_{\lesssim}$ (resp.\ $\min_{\lesssim}$) defined on $X^2\setminus\{(x,y)\in X^2: x\sim y,~x\neq y\}$ by $\max_{\lesssim}(x,y)=y$ (resp.\ $\min_{\lesssim}(x,y)=x$) whenever $x\lesssim y$. We observe that if $\lesssim$ reduces to a total ordering, then the operation $\max_{\lesssim}$ (resp.\ $\min_{\lesssim}$) is defined everywhere on $X^2$.

Also, the \emph{projection operations} $\pi_1\colon X^2\to X$ and $\pi_2\colon X^2\to X$ (also called \emph{left-} and \emph{right-semigroups}) are respectively defined by $\pi_1(x,y)=x$ and $\pi_2(x,y)=y$ for all $x,y\in X$.

An element $e\in X$ is said to be a \emph{neutral element} of $F\colon X^2\to X$ if $F(x,e)=F(e,x)=x$ for all $x\in X$. An element $a\in X$ is said to be an \emph{annihilator element} of $F\colon X^2\to X$ if $F(x,a)=F(a,x)=a$ for all $x\in X$.

For any integer $n\geq 1$, any $F\colon X_n^2\to X_n$, and any $z\in X_n$, the \emph{$F$-degree of $z$}, denoted $\deg_F(z)$, is the number of points $(x,y)\in X_n^2\setminus\{(z,z)\}$ such that $F(x,y)=F(z,z)$. Also, the \emph{degree sequence of $F$}, denoted $\deg_F$, is the nondecreasing $n$-element sequence of the numbers $\deg_F(x)$, $x\in X_n$.

We now state the key theorem on which most of our results rely. It provides a descriptive characterization of the class of associative and quasitrivial operations on $X$. As observed by Ackerman \cite[Section~1.2]{Ack}, this result is a simple consequence of two papers on idempotent semigroups, namely Kimura~\cite{Kim58} and McLean~\cite{McL54}. It was also independently presented by various authors (see, e.g., Kepka~\cite[Corollary~1.6]{Kep81} and L\"anger~\cite[Theorem~1]{Lan80}). For the sake of completeness we provide a direct elementary proof.

\begin{theorem}\label{thm:kimura}
$F\colon X^2\to X$ is associative and quasitrivial if and only if there exists a weak ordering $\precsim$ on $X$ such that
\begin{equation}\label{eq:kimura}
F|_{A\times B} ~=~
\begin{cases}
\max_{\precsim}|_{A\times B}, & \text{if $A\neq B$},\\
\pi_1|_{A\times B}\hspace{1.5ex}\text{or}\hspace{1.5ex}\pi_2|_{A\times B}, & \text{if $A=B$},
\end{cases}
\qquad \forall A,B\in {X/\sim}.
\end{equation}
\end{theorem}

\begin{proof}
(Sufficiency) Trivial.

(Necessity) We observe that the binary relation $\precsim$ defined on $X$ by
\begin{equation}\label{eq:prekimura}
x\precsim y \quad\Leftrightarrow\quad F(x,y)=y\hspace{1ex}\text{or}\hspace{1ex}F(y,x)=y,\qquad x,y\in X,
\end{equation}
is a weak ordering on $X$. Indeed, this relation is clearly total. Let us show that it is transitive. Let $x,y,z\in X$ be pairwise distinct and such that $x\precsim y$ and $y\precsim z$. Let us assume for instance that $F(x,y)=y$ and $F(z,y)=z$ (the other three cases can be dealt with similarly). Then we have $F(x,z)=z$ and hence $x\precsim z$. Indeed, otherwise we would have $x=F(x,z)=F(x,F(z,y))=F(F(x,z),y)=F(x,y)=y$, a contradiction.

Let us now show that Eq.~\eqref{eq:kimura} holds. It is easy to see that for any $x,y\in X$ such that $x\prec y$, we have $F|_{\{x,y\}^2}=\textstyle{\max_{\precsim}|_{\{x,y\}^2}}$. Similarly, for any distinct $x,y\in X$ such that $x\sim y$ we have $F|_{\{x,y\}^2}=\pi_1|_{\{x,y\}^2}$ or $F|_{\{x,y\}^2}=\pi_2|_{\{x,y\}^2}$. Finally, let us show that for any pairwise distinct $x,y,z\in X$ such that $x\sim y\sim z$, we cannot have both $F|_{\{x,y\}^2}=\pi_1|_{\{x,y\}^2}$ and $F|_{\{x,z\}^2}=\pi_2|_{\{x,z\}^2}$. Indeed, otherwise
\begin{itemize}
\item if $F(y,z)=y$, then $z=F(x,z)=F(F(x,y),z)=F(x,F(y,z))=F(x,y)=x$,
\item if $F(y,z)=z$, then $y=F(y,x)=F(y,F(z,x))=F(F(y,z),x)=F(z,x)=x$.
\end{itemize}
We reach a contradiction in each of these cases.
\end{proof}

It is not difficult to see that the weak ordering $\precsim$ mentioned in Theorem~\ref{thm:kimura} is uniquely determined from $F$ and can be defined by condition \eqref{eq:prekimura}. If $X=X_n$ for some integer $n\geq 1$, then $\precsim$ can be as well defined as follows: $x\precsim y$ if and only if $\deg_F(x)\leq\deg_F(y)$.\footnote{Thus, when $X=X_n$ the weak ordering $\precsim$ is completely determined by a set of $n$ integers (actually $n-1$ integers since we have $\sum_{x\in X_n}\deg_F(x)=n(n-1)$ whenever $F$ is idempotent).} This latter equivalence can be easily derived (see Corollary~\ref{cor:degdeg}) from the following proposition.

\begin{proposition}\label{prop:degxe}
If $F\colon X_n^2\to X_n$ is of the form \eqref{eq:kimura} for some weak ordering $\precsim$ on $X_n$, then for any $x\in X_n$, we have
\begin{eqnarray*}
\deg_F(x) &=& 2\times|\{z\in X_n: z\prec x\}|+|\{z\in X_n: z\sim x,~z\neq x\}|\\
&=& |\{z\in X_n: z\prec x\}|+|\{z\in X_n: z\precsim x\}|-1.
\end{eqnarray*}
\end{proposition}

\begin{proof}
Let $x\in X_n$. By quasitriviality, only points of the form $(x,z)$ or $(z,x)$, with $z\in X_n$, may have the same value as $(x,x)$.
\begin{itemize}
\item If $z\prec x$, then $F(x,z)=F(z,x)=x=F(x,x)$.
\item If $x\prec z$, then $F(x,z)=F(z,x)=z\neq F(x,x)$.
\item If $z\sim x$ and $z\neq x$, then either $F(x,z)=\pi_1(x,z)$ or $F(x,z)=\pi_2(x,z)$. In the first case, we have $F(x,z)=x=F(x,x)\neq z=F(z,x)$. The other case is similar.
\end{itemize}
This completes the proof of Proposition~\ref{prop:degxe}.
\end{proof}

\begin{corollary}\label{cor:degdeg}
If $F\colon X_n^2\to X_n$ is of the form \eqref{eq:kimura} for some weak ordering $\precsim$ on $X_n$, then for any $x,y\in X_n$, we have
$$
x\precsim y\quad\Leftrightarrow\quad\deg_F(x)\leq\deg_F(y).
$$
\end{corollary}

\begin{proof}
Let $x,y\in X_n$ such that $x\precsim y$. We clearly have
$$
|\{z\in X_n: z\prec x\}|~\leq ~|\{z\in X_n: z\prec y\}|
$$
and
$$
|\{z\in X_n: z\precsim x\}|~\leq ~|\{z\in X_n: z\precsim y\}|.
$$
By Proposition~\ref{prop:degxe}, we then immediately have $\deg_F(x)\leq\deg_F(y)$. The (contrapositive of the) reverse implication can be proved similarly.
\end{proof}

From the properties of the maximum operation in \eqref{eq:kimura}, we can observe the following fact.

\begin{fact}\label{fact:2}
If $F\colon X^2\to X$ is of the form \eqref{eq:kimura} for some weak ordering $\precsim$ on $X$, then $F$ has a neutral element $e\in X$ (resp.\ an annihilator element $a\in X$) if and only if the weakly ordered set $(X,\precsim)$ has a unique minimal element denoted by $x^{\bot}$ (resp.\ a unique maximal element denoted by $x^{\top}$). In this case we have $e = x^{\bot}$ (resp.\ $a=x^{\top}$).
\end{fact}

\begin{remark}
If $F\colon X^2\to X$ is of the form \eqref{eq:kimura} for some weak ordering $\precsim$ on $X$, then, by replacing $\precsim$ with its inverse relation ${\precsim^{-1}}$ (defined by ${a\precsim^{-1}}b\Leftrightarrow b\precsim a$), we see that $F$ is again of the form \eqref{eq:kimura}, except that the maximum operation is changed to the minimum operation. Thus, choosing the maximum or the minimum operation is just a matter of convention.
\end{remark}

The following lemma can be obtained by following the first few steps of the proof of \cite[Theorem~3]{CzoDre84}, which was stated in the special case where $X$ is an arbitrary closed real interval. For the sake of self-containedness we provide a short proof.

\begin{lemma}[{see \cite[Theorem~3]{CzoDre84}}]\label{lemma:C-D}
If $F\colon X^2\to X$ is associative, idempotent, $\leq$-preserving for some total ordering $\leq$ on $X$, and has a neutral element, then $F$ is quasitrivial.
\end{lemma}

\begin{proof}
Let $e$ denote the neutral element of $F$. By idempotency and $\leq$-preservation we clearly have $\min_{\leq}(x,y)\leq F(x,y)\leq \max_{\leq}(x,y)$ for all $x,y\in X$. If $x,y\leq e$, then by $\leq$-preservation we obtain $F(x,y)\leq \min_{\leq}(F(x,e),F(e,y))=\min_{\leq}(x,y)$. Thus $F(x,y)=\min_{\leq}(x,y)$ whenever $x,y\leq e$. We show dually that $F(x,y)=\max_{\leq}(x,y)$ whenever $x,y\geq e$. Assume now that $x<e<y$ (the case $y<e<x$ can be dealt with dually). If $F(x,y)\leq e$, then $F(x,y)=F(F(x,x),y)=F(x,F(x,y))=\min_{\leq}(x,F(x,y))=x$. We prove similarly that $F(x,y)=y$ whenever $F(x,y)\geq e$. It follows that $F$ is quasitrivial.
\end{proof}

When $X_n$ is endowed with $\leq_n$, the operations $F\colon X_n^2\to X_n$ can be visualized through their contour plots, where we connect points in $X_n^2$ having the same $F$-values by edges or paths. For instance, the operation $F\colon X_6^2\to X_6$ whose contour plot is shown in Figure~\ref{fig:fs2} is associative, quasitrivial, commutative, and $\leq_6$-preserving.

\setlength{\unitlength}{3.5ex}
\begin{figure}[htbp]
\begin{center}
\begin{small}
\begin{picture}(8,8)
\put(0.5,0.5){\vector(1,0){7}}\put(0.5,0.5){\vector(0,1){7}}
\multiput(1.5,0.45)(1,0){6}{\line(0,1){0.1}}%
\multiput(0.45,1.5)(0,1){6}{\line(1,0){0.1}}%
\put(1.5,0){\makebox(0,0){$1$}}\put(2.5,0){\makebox(0,0){$2$}}\put(3.5,0){\makebox(0,0){$3$}}
\put(4.5,0){\makebox(0,0){$4$}}\put(5.5,0){\makebox(0,0){$5$}}\put(6.5,0){\makebox(0,0){$6$}}
\put(0,1.5){\makebox(0,0){$1$}}\put(0,2.5){\makebox(0,0){$2$}}\put(0,3.5){\makebox(0,0){$3$}}
\put(0,4.5){\makebox(0,0){$4$}}\put(0,5.5){\makebox(0,0){$5$}}\put(0,6.5){\makebox(0,0){$6$}}
\multiput(1.5,1.5)(0,1){6}{\multiput(0,0)(1,0){6}{\circle*{0.2}}}
\drawline[1](6.5,1.5)(1.5,1.5)(1.5,6.5)\drawline[1](2.5,6.5)(6.5,6.5)(6.5,2.5)\drawline[1](2.5,5.5)(2.5,2.5)(5.5,2.5)
\drawline[1](3.5,5.5)(3.5,3.5)(5.5,3.5)\drawline[1](4.5,5.5)(5.5,5.5)(5.5,4.5)
\put(1.85,1.85){\makebox(0,0){$1$}}\put(2.85,2.85){\makebox(0,0){$2$}}\put(3.85,3.85){\makebox(0,0){$3$}}
\put(4.85,4.85){\makebox(0,0){$4$}}\put(5.85,5.85){\makebox(0,0){$5$}}\put(6.85,6.85){\makebox(0,0){$6$}}
\end{picture}
\end{small}
\caption{An associative and quasitrivial operation on $X_6$ (contour plot)}
\label{fig:fs2}
\end{center}
\end{figure}
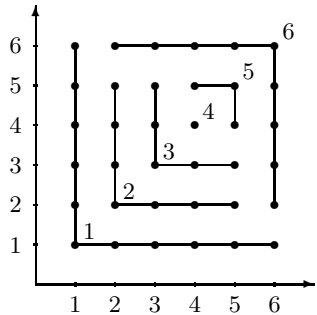

Two points $(x,y)$ and $(u,v)$ of $X_n^2$ are said to be \emph{$F$-connected} if they have the same $F$-value, i.e., if $F(x,y)=F(u,v)$. Using this definition, we can state the following four graphical tests (see \cite{CouDevMar}), where $F\colon X_n^2\to X_n$ denotes an arbitrary operation and $\Delta_{X_n}$ denotes the set $\{(x,x): x\in X_n\}$.
\begin{itemize}
\item $F$ is quasitrivial if and only if it is idempotent and every point $(x,y)\in X_n^2\setminus\Delta_{X_n}$ is $F$-connected to either $(x,x)$ or $(y,y)$.
\item If $F$ is quasitrivial, then $e\in X_n$ is a neutral element of $F$ if and only if the point $(e,e)$ is not $F$-connected to another point, i.e., if and only if $\deg_F(e)=0$.
\item If $F$ is quasitrivial, then $a\in X_n$ is an annihilator element of $F$ if and only if the point $(a,a)$ is $F$-connected to exactly $2n-2$ points, i.e., if and only if $\deg_F(a)=2n-2$.
\item If $F$ is quasitrivial, then it is associative if and only if for every rectangle in $X_n^2$ that has only one vertex on $\Delta_{X_n}$, at least two of the remaining three vertices are $F$-connected.
\end{itemize}

\section{Quasitrivial and commutative semigroups}

In this section we provide characterizations of the class of associative, quasitrivial, and commutative operations $F\colon X^2\to X$, or equivalently, the class of quasitrivial and commutative semigroups on $X$. We also characterize the subclass of those operations that are order-preserving with respect to some total ordering on $X$.

The first characterization is given in the following theorem, which immediately follows from  Theorem~\ref{thm:kimura}. We observe that Ackerman (see \cite[Corollary~4.10]{Ack}) generalized this result to $n$-ary semigroups for any integer $n\geq 2$.

\begin{theorem}\label{thm:mainSy}
$F\colon X^2\to X$ is associative, quasitrivial, and commutative if and only if there exists a total ordering $\preceq$ on $X$ such that $F=\max_{\preceq}$.
\end{theorem}

Theorem~\ref{thm:char1} below provides alternative characterizations of the class of associative, quasitrivial, and commutative operations. We first consider the following auxiliary lemma.

\begin{lemma}\label{lemma:QiA}
If $F\colon X^2\to X$ is quasitrivial, commutative, $\leq$-preserving for some total ordering $\leq$ on $X$, then $F$ is associative.
\end{lemma}

\begin{proof}
This result was established in the special case where $X$ is the real unit interval $[0,1]$ in \cite[Proposition~2]{MarMayTor03}. The proof therein is purely algebraic and hence it applies to any nonempty totally ordered set.
\end{proof}

\begin{theorem}\label{thm:char1}
Let $F\colon X^2\to X$ be an operation. The following assertions are equivalent.
\begin{enumerate}
\item[(i)] $F$ is associative, quasitrivial, and commutative.
\item[(ii)] $F=\max_{\preceq}$ for some total ordering $\preceq$ on $X$.
\item[(iii)] $F$ is quasitrivial, commutative, and $\leq$-preserving for some total ordering $\leq$ on $X$.
\end{enumerate}
If $X=X_n$ for some integer $n\geq 1$, then any of the assertions (i)--(iii) above is equivalent to any of the following ones.
\begin{enumerate}
\item[(iv)] $F$ is quasitrivial and satisfies $\deg_F=(0,2,4,\ldots,2n-2)$.
\item[(v)] $F$ is associative, idempotent, commutative, $\leq$-preserving for some total ordering $\leq$ on $X$, and has a neutral element.
\end{enumerate}
Moreover, there are exactly $n!$ operations $F\colon X_n^2\to X_n$ satisfying any of the assertions (i)--(v). Furthermore, the total ordering $\preceq$ considered in assertion (ii) is uniquely defined as follows: $x\preceq y$ if and only if $\deg_F(x)\leq\deg_F(y)$. In particular, each of these operations has the (unique) neutral element $e=\min_{\preceq}X_n$ and the (unique) annihilator element $a=\max_{\preceq}X_n$.
\end{theorem}

\begin{proof}
The equivalence (i) $\Leftrightarrow$ (ii) $\Leftrightarrow$ (iii) follows from Theorem~\ref{thm:mainSy} and Lemma~\ref{lemma:QiA}. We have (ii) $\Rightarrow$ (v) by Fact~\ref{fact:2} and (v) $\Rightarrow$ (iii) by Lemma~\ref{lemma:C-D}. Also, it is clear that (ii) $\Rightarrow$ (iv).

Let us now show by induction on $n$ that (iv) $\Rightarrow$ (ii). The result clearly holds for $n=1$. Suppose that it holds for some $n\geq 1$ and let us show that it still holds for $n+1$. Assume that $F\colon X_{n+1}^2\to X_{n+1}$ is quasitrivial and that $\deg_F=(0,2,\ldots,2n)$. Let $\preceq$ be the unique total ordering on $X_{n+1}$ defined by $x\preceq y$ if and only if $\deg_F(x)\leq\deg_F(y)$ and let $z=\max_{\preceq}X_{n+1}$. Clearly, the operation $F'=F|_{(X_{n+1}\setminus\{z\})^2}$ is quasitrivial and such that $\deg_{F'}=(0,2,\ldots,2n-2)$. By induction hypothesis we have $F'=\max_{\preceq'}$, where $\preceq'$ is the restriction of $\preceq$ to $(X_{n+1}\setminus\{z\})^2$. Since $\deg_F(z)=2n$ we necessarily have $F=\max_{\preceq}$.

To complete the proof of the theorem, we observe that there are exactly $n!$ total orderings on $X_n$ and hence exactly $n!$ operations $F\colon X_n^2\to X_n$ satisfying assertion (ii). The rest of the statement is immediate.
\end{proof}

\begin{remark}
The existence of a neutral element in assertion (v) of Theorem~\ref{thm:char1} cannot be replaced with the existence of an annihilator element. Indeed, the operation $F\colon X_3^2\to X_3$ whose contour plot is depicted in Figure~\ref{fig:re2} is associative, idempotent, commutative, $\leq_3$-preserving, and has the annihilator element $a=2$. However it is not quasitrivial.
\end{remark}

\setlength{\unitlength}{5ex}
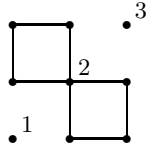
\begin{figure}[htbp]
\begin{center}
\begin{small}
\begin{picture}(3,3)
\multiput(0.5,0.5)(0,1){3}{\multiput(0,0)(1,0){3}{\circle*{0.15}}}
\drawline[1](0.5,2.5)(1.5,2.5)(1.5,0.5)(2.5,0.5)(2.5,1.5)(0.5,1.5)(0.5,2.5)
\put(0.75,0.75){\makebox(0,0){$1$}}\put(1.75,1.75){\makebox(0,0){$2$}}\put(2.75,2.75){\makebox(0,0){$3$}}
\end{picture}
\end{small}
\caption{An associative operation on $X_3$ that is not quasitrivial}
\label{fig:re2}
\end{center}
\end{figure}

We now consider the subclass of associative, quasitrivial, and commutative operations $F\colon X^2\to X$ that are $\leq$-preserving for some fixed total ordering $\leq$ on $X$. To this extent we recall the single-peakedness property for arbitrary totally ordered sets. This notion was first introduced for finite totally ordered sets (i.e., finite chains) in social choice theory by Black~\cite{Bla48,Bla87}.

\begin{definition}[{see \cite[Definition 3.8]{DevKisMar}}]\label{de:Black}
Let $\leq$ and $\preceq$ be total orderings on $X$. We say that $\preceq$ is \emph{single-peaked for $\leq$} if for any $a,b,c\in X$ such that $a<b<c$, we have $b\prec a$ or $b\prec c$.
\end{definition}

\begin{example}
There are four total orderings $\preceq$ on $X_3$ that are single-peaked for $\leq_3$, namely $1\prec 2\prec 3$, $2\prec 1\prec 3$, $2\prec 3\prec 1$, and $3\prec 2\prec 1$.
\end{example}

For arbitrary total orderings $\leq$ and $\preceq$ on $X$, the operation $F=\max_{\preceq}$ need not be $\leq$-preserving. The following proposition characterizes those total orderings $\preceq$ on $X$ for which $F=\max_{\preceq}$ is $\leq$-preserving.

\begin{proposition}[{see \cite[Proposition 3.9]{DevKisMar}}]\label{prop:Qspl}
Let $\leq$ be a total ordering on $X$ and let $F\colon X^2\to X$ be given by $F=\max_{\preceq}$ for some total ordering $\preceq$ on $X$. Then $F$ is $\leq$-preserving if and only if $\preceq$ is single-peaked for $\leq$.
\end{proposition}

The following theorem is an immediate consequence of Lemma~\ref{lemma:QiA}, Theorem~\ref{thm:char1}, Proposition~\ref{prop:Qspl} and the known fact (see also Section~\ref{sec:5}) that there are exactly $2^{n-1}$ total orderings on $X_n$ that are single-peaked for $\leq_n$.

\begin{theorem}\label{thm:cchar1}
Let $F\colon X^2\to X$ be an operation and let $\leq$ be a total ordering on $X$. The following assertions are equivalent.
\begin{enumerate}
\item[(i)] $F$ is quasitrivial, commutative, and $\leq$-preserving.
\item[(ii)] $F=\max_{\preceq}$ for some total ordering $\preceq$ on $X$ that is single-peaked for $\leq$.
\end{enumerate}
If $(X,\leq)=(X_n,\leq_n)$ for some integer $n\geq 1$, then any of the assertions (i)--(ii) above is equivalent to any of the following ones.
\begin{enumerate}
\item[(iii)] $F$ is quasitrivial, $\leq$-preserving, and satisfies $\deg_F=(0,2,4,\ldots,2n-2)$.
\item[(iv)] $F$ is associative, idempotent, commutative, $\leq$-preserving, and has a neutral element.
\end{enumerate}
Moreover, there are exactly $2^{n-1}$ operations $F\colon X_n^2\to X_n$ satisfying any of the assertions (i)--(iv). Furthermore, the total ordering $\preceq$ considered in assertion (ii) is uniquely defined as follows: $x\preceq y$ if and only if $\deg_F(x)\leq\deg_F(y)$. In particular, each of these operations has the (unique) neutral element $e=\min_{\preceq}X_n$ and the (unique) annihilator element $a=\max_{\preceq}X_n$.
\end{theorem}

\begin{example}
In Figure~\ref{fig:avre3} we present the $3!=6$ associative, quasitrivial, and commutative operations on $X_3$. Only the first $2^{3-1}=4$ operations are $\leq_3$-preserving. All these operations have neutral and annihilator elements.
\end{example}

\setlength{\unitlength}{3.5ex}
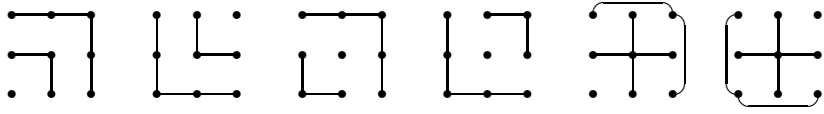
\begin{figure}[htbp]
\begin{center}
\begin{small}
\begin{picture}(3,3)
\multiput(0.5,0.5)(0,1){3}{\multiput(0,0)(1,0){3}{\circle*{0.2}}}
\drawline[1](0.5,2.5)(2.5,2.5)(2.5,0.5)\drawline[1](0.5,1.5)(1.5,1.5)(1.5,0.5)
\end{picture}
\hspace{0.01\textwidth}
\begin{picture}(3,3)
\multiput(0.5,0.5)(0,1){3}{\multiput(0,0)(1,0){3}{\circle*{0.2}}}
\drawline[1](0.5,2.5)(0.5,0.5)(2.5,0.5)\drawline[1](1.5,2.5)(1.5,1.5)(2.5,1.5)
\end{picture}
\hspace{0.01\textwidth}
\begin{picture}(3,3)
\multiput(0.5,0.5)(0,1){3}{\multiput(0,0)(1,0){3}{\circle*{0.2}}}
\drawline[1](0.5,2.5)(2.5,2.5)(2.5,0.5)\drawline[1](0.5,1.5)(0.5,0.5)(1.5,0.5)
\end{picture}
\hspace{0.01\textwidth}
\begin{picture}(3,3)
\multiput(0.5,0.5)(0,1){3}{\multiput(0,0)(1,0){3}{\circle*{0.2}}}
\drawline[1](0.5,2.5)(0.5,0.5)(2.5,0.5)\drawline[1](1.5,2.5)(2.5,2.5)(2.5,1.5)
\end{picture}
\hspace{0.01\textwidth}
\begin{picture}(3,3)
\multiput(0.5,0.5)(0,1){3}{\multiput(0,0)(1,0){3}{\circle*{0.2}}}
\drawline[1](0.5,1.5)(2.5,1.5)\drawline[1](1.5,0.5)(1.5,2.5)
\put(2.5,1.5){\oval(0.6,2)[r]}\put(1.5,2.5){\oval(2,0.6)[t]}
\end{picture}
\hspace{0.01\textwidth}
\begin{picture}(3,3)
\multiput(0.5,0.5)(0,1){3}{\multiput(0,0)(1,0){3}{\circle*{0.2}}}
\drawline[1](0.5,1.5)(2.5,1.5)\drawline[1](1.5,0.5)(1.5,2.5)
\put(0.5,1.5){\oval(0.6,2)[l]}\put(1.5,0.5){\oval(2,0.6)[b]}
\end{picture}
\end{small}
\caption{The six associative, quasitrivial, and commutative operations on $X_3$}
\label{fig:avre3}
\end{center}
\end{figure}

\begin{remark}\label{rem:3ab}
\begin{enumerate}
\item[(a)] To better illustrate Theorem~\ref{thm:cchar1} when $X$ is finite, consider the operation $F\colon X_6^2\to X_6$ whose contour plot is shown in Figure~\ref{fig:r3a} (left). This operation is clearly quasitrivial, $\leq_6$-preserving, and is such that $\deg_F=(0,2,\ldots,10)$. By Theorem~\ref{thm:cchar1} we then have $F=\max_{\preceq}$, where $\preceq$ is the total ordering on $X_6$ obtained by sorting the numbers $\deg_F(x)$, $x\in X_6$, in increasing order, that is, $4\prec 3\prec 5\prec 2\prec 1\prec 6$; see Figure~\ref{fig:r3a} (right). This total ordering is single-peaked for $\leq_6$ (see also Example~\ref{ex:2fl}).
\item[(b)] The equivalence between assertions (i) and (ii) of Theorem~\ref{thm:cchar1} was established in \cite[Theorem 3.13]{DevKisMar}. When $X$ is finite, the equivalence among assertions (i), (ii), and (iv) of Theorem~\ref{thm:cchar1} was established in \cite[Theorems 12 and 17]{CouDevMar}.
\end{enumerate}
\end{remark}

\setlength{\unitlength}{3.5ex}
\begin{figure}[htbp]
\begin{center}
\begin{small}
\null\hspace{0.03\textwidth}
\begin{picture}(8,8)
\put(0.5,0.5){\vector(1,0){7}}\put(0.5,0.5){\vector(0,1){7}}
\multiput(1.5,0.45)(1,0){6}{\line(0,1){0.1}}%
\multiput(0.45,1.5)(0,1){6}{\line(1,0){0.1}}%
\put(1.5,0){\makebox(0,0){$1$}}\put(2.5,0){\makebox(0,0){$2$}}\put(3.5,0){\makebox(0,0){$3$}}
\put(4.5,0){\makebox(0,0){$4$}}\put(5.5,0){\makebox(0,0){$5$}}\put(6.5,0){\makebox(0,0){$6$}}
\put(0,1.5){\makebox(0,0){$1$}}\put(0,2.5){\makebox(0,0){$2$}}\put(0,3.5){\makebox(0,0){$3$}}
\put(0,4.5){\makebox(0,0){$4$}}\put(0,5.5){\makebox(0,0){$5$}}\put(0,6.5){\makebox(0,0){$6$}}
\put(2,0){\makebox(0,0){$<$}}\put(3,0){\makebox(0,0){$<$}}\put(4,0){\makebox(0,0){$<$}}
\put(5,0){\makebox(0,0){$<$}}\put(6,0){\makebox(0,0){$<$}}
\multiput(1.5,1.5)(0,1){6}{\multiput(0,0)(1,0){6}{\circle*{0.2}}}
\drawline[1](1.5,6.5)(6.5,6.5)(6.5,1.5)\drawline[1](1.5,5.5)(1.5,1.5)(5.5,1.5)\drawline[1](2.5,5.5)(2.5,2.5)(5.5,2.5)
\drawline[1](3.5,5.5)(5.5,5.5)(5.5,3.5)\drawline[1](3.5,4.5)(3.5,3.5)(4.5,3.5)
\end{picture}
\hspace{0.1\textwidth}
\begin{picture}(8,8)
\put(0.5,0.5){\vector(1,0){7}}\put(0.5,0.5){\vector(0,1){7}}
\multiput(1.5,0.45)(1,0){6}{\line(0,1){0.1}}%
\multiput(0.45,1.5)(0,1){6}{\line(1,0){0.1}}%
\put(1.5,0){\makebox(0,0){$4$}}\put(2.5,0){\makebox(0,0){$3$}}\put(3.5,0){\makebox(0,0){$5$}}
\put(4.5,0){\makebox(0,0){$2$}}\put(5.5,0){\makebox(0,0){$1$}}\put(6.5,0){\makebox(0,0){$6$}}
\put(0,1.5){\makebox(0,0){$4$}}\put(0,2.5){\makebox(0,0){$3$}}\put(0,3.5){\makebox(0,0){$5$}}
\put(0,4.5){\makebox(0,0){$2$}}\put(0,5.5){\makebox(0,0){$1$}}\put(0,6.5){\makebox(0,0){$6$}}
\put(2,0){\makebox(0,0){$\prec$}}\put(3,0){\makebox(0,0){$\prec$}}\put(4,0){\makebox(0,0){$\prec$}}
\put(5,0){\makebox(0,0){$\prec$}}\put(6,0){\makebox(0,0){$\prec$}}
\multiput(1.5,1.5)(0,1){6}{\multiput(0,0)(1,0){6}{\circle*{0.2}}}
\drawline[1](1.5,6.5)(6.5,6.5)(6.5,1.5)\drawline[1](1.5,5.5)(5.5,5.5)(5.5,1.5)\drawline[1](1.5,4.5)(4.5,4.5)(4.5,1.5)
\drawline[1](1.5,3.5)(3.5,3.5)(3.5,1.5)\drawline[1](1.5,2.5)(2.5,2.5)(2.5,1.5)
\end{picture}
\end{small}
\caption{An operation $F\colon X_6^2\to X_6$ defined by $F=\max_{\preceq}$, where $\preceq$ is single-peaked for $\leq_6$}
\label{fig:r3a}
\end{center}
\end{figure}
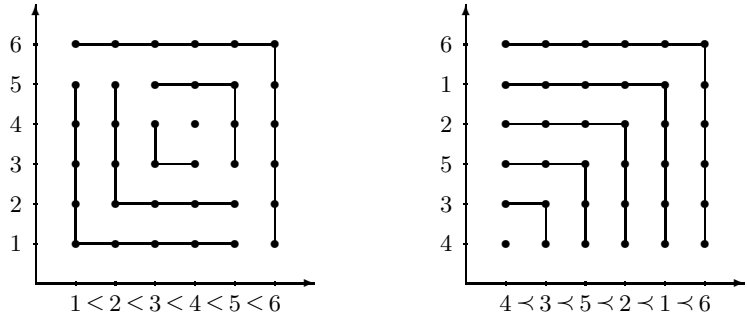

\section{Enumerations of arbitrary quasitrivial semigroups}

This section is devoted to the arbitrary associative and quasitrivial operations that need not be commutative. Recall that a characterization of this class of operations is given in Theorem~\ref{thm:kimura}. However, to our knowledge a generalization of Theorem~\ref{thm:char1} to noncommutative operations is not known and hence remains an open problem. On this issue we make the following two observations.
\begin{itemize}
\item An associative and quasitrivial operation $F\colon X^2\to X$ need not have a neutral element, even if $X$ is finite. For instance, the projection operations $\pi_1$ and $\pi_2$ have no neutral element.
\item An associative and quasitrivial operation $F\colon X^2\to X$ need not be $\leq$-preserving for some total ordering $\leq$ on $X$, even if $X$ is finite. To illustrate, consider $F\colon X_4^2\to X_4$ whose contour plot is depicted in Figure~\ref{fig:nop}. This operation is associative and quasitrivial. However, it can be shown that it is not $\leq$-preserving for any of the 24 total orderings $\leq$ on $X_4$.
\end{itemize}

\setlength{\unitlength}{3.5ex}
\begin{figure}[htbp]
\begin{center}
\begin{small}
\begin{picture}(4,4)
\multiput(0.5,0.5)(0,1){4}{\multiput(0,0)(1,0){4}{\circle*{0.2}}}
\drawline[1](0.5,3.5)(3.5,3.5)\drawline[1](0.5,2.5)(3.5,2.5)\drawline[1](2.5,2.5)(2.5,1.5)
\drawline[1](0.5,1.5)(0.5,0.5)(3.5,0.5)\put(3.5,2.5){\oval(0.6,2)[r]}
\end{picture}
\end{small}
\caption{An operation that is not $\leq$-preserving for any total ordering $\leq$}
\label{fig:nop}
\end{center}
\end{figure}
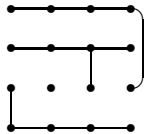

In the rest of this section we consider the problem of enumerating quasitrivial semigroups on finite sets. For instance, for any integer $n\geq 1$, we provide in Theorem~\ref{thm:Nk} the exact number of associative and quasitrivial operations $F\colon X_n^2\to X_n$. We posted the corresponding sequence in Sloane's On-Line Encyclopedia of Integer Sequences (OEIS, see \cite{Slo}) as sequence A292932.

In this section we often consider either the (ordinary) generating function (GF) or the exponential generating function (EGF) of a given integer sequence $(s_n)_{n\geq 0}$. Recall that, when these functions exist, they are respectively defined by the power series
$$
S(z) ~=~ \sum_{n\geq 0}s_n{\,}z^n\quad\text{and}\quad \hat S(z) ~=~ \sum_{n\geq 0}s_n{\,}\frac{z^n}{n!}{\,}.
$$

Recall also that for any integers $0\leq k\leq n$, the \emph{Stirling number of the second kind} ${n\brace k}$ is defined by
$$
{n\brace k} ~=~ \frac{1}{k!}{\,}\sum_{i=0}^k(-1)^{k-i}{k\choose i}{\,}i^{n}.
$$

For any integer $n\geq 0$, let $p(n)$ denote the number of weak orderings on $X_n$, or equivalently, the number of ordered partitions of $X_n$. Setting $p(0)=1$, the number $p(n)$ is explicitly given by
$$
p(n) ~=~ \sum_{k=0}^n{n\brace k}{\,}k!{\,},\qquad n\geq 0.
$$
Actually, the corresponding sequence $(p(n))_{n\geq 0}$ consists of the \emph{ordered Bell numbers} (Sloane's A000670) and satisfies the following recurrence equation
$$
p(n+1) ~=~ \sum_{k=0}^n{n+1\choose k}{\,}p(k){\,},\qquad n\geq 0,
$$
with $p(0)=1$. Moreover, its EGF is given by $\hat P(z) = 1/(2-e^z)$.

For any integer $n\geq 1$, we denote by $q(n)$ the number of associative and quasitrivial operations $F\colon X_n^2\to X_n$ (i.e., the number of quasitrivial semigroups on an $n$-element set). As a convention, we set $q(0)=1$. Also, for any integer $n\geq 0$, we denote by
\begin{itemize}
\item $q_e(n)$ the number of associative and quasitrivial operations $F\colon X_n^2\to X_n$ that have neutral elements,
\item $q_a(n)$ the number of associative and quasitrivial operations $F\colon X_n^2\to X_n$ that have annihilator elements,
\item $q_{ea}(n)$ the number of associative and quasitrivial operations $F\colon X_n^2\to X_n$ that have distinct neutral and annihilator elements.
\end{itemize}
As a convention, we set $q_e(0)=q_a(0)=q_{ea}(0)=0$. Theorem~\ref{thm:Nk} and Proposition~\ref{prop:qe} below provide explicit formulas for these sequences. The first few values of these sequences are shown in Table~\ref{tab:q}.

\begin{theorem}\label{thm:Nk}
For any integer $n\geq 0$, we have the closed-form expression
\begin{equation}\label{eq:Nk}
q(n) ~=~ \sum_{i=0}^n 2^i{\,}\sum_{k=0}^{n-i}(-1)^k{\,}{n\choose k}{n-k\brace i}{\,}(i+k)!{\,},\qquad n\geq 0.
\end{equation}
Moreover, the sequence $(q(n))_{n\geq 0}$ satisfies the recurrence equation
$$
q(n+1) ~=~ (n+1){\,}q(n)+2\sum_{k=0}^{n-1}{n+1\choose k}{\,}q(k){\,},\qquad n\geq 0,
$$
with $q(0)=1$. Furthermore, its EGF is given by $\hat Q(z)=1/(z+3-2e^z)$.
\end{theorem}

\begin{proof}
Using Theorem~\ref{thm:kimura} we can easily see that
\begin{equation}\label{eq:Inqn}
q(n) ~=~ \sum_{k=1}^n ~\sum_{\textstyle{n_1,\ldots,n_k\geq 1\atop n_1+\cdots +n_k=n}}{n\choose n_1,\ldots,n_k}\prod_{\textstyle{i=1\atop n_i\geq 2}}^k 2{\,},\qquad n\geq 1.
\end{equation}
Indeed, to compute $q(n)$ we need to consider all the ordered partitions of $X_n$ and count twice each equivalence class containing at least two elements (because two possible projections are to be considered for each such class). In Eq.~\eqref{eq:Inqn}, $k$ represents the number of equivalence classes and $n_i$ represents the cardinality of the $i$th class.

For any integer $k\geq 1$, define the sequence $(s_n^k)_{n\geq 0}$ as
\begin{equation}\label{eq:DeSkn}
s_n^k ~=~ \sum_{\textstyle{n_1,\ldots,n_k\geq 0\atop n_1+\cdots +n_k=n}}{n\choose n_1,\ldots,n_k}\prod_{i=1}^k \min\{n_i,2\}.
\end{equation}
Thus defined, the sequence $(s_n^k)_{n\geq 0}$ is the $k$-fold binomial convolution of the sequence $(\min\{n,2\})_{n\geq 0}$ (for background on convolutions see, e.g., \cite[Section 7.2.1]{MarTo16}). Since the EGF of the latter sequence is clearly the function $z\mapsto 2e^z-z-2$, it follows that the EGF of the sequence $(s_n^k)_{n\geq 0}$ is the function $z\mapsto (2e^z-z-2)^k$, which means that
\begin{equation}\label{eq:snkD}
s_n^k ~=~ D_z^n (2e^z-z-2)^k|_{z=0}{\,},
\end{equation}
where $D_z$ denotes the usual differential operator.

Using \eqref{eq:Inqn}--\eqref{eq:snkD}, for any integer $n\geq 1$, we then obtain
$$
q(n) ~=~ \sum_{k=1}^n s_n^k ~=~ D_z^n{\,}\frac{1-(2e^z-z-2)^{n+1}}{z+3-2e^z}\Big|_{z=0} ~=~ D_z^n{\,}\frac{1}{z+3-2e^z}\Big|_{z=0}.
$$
Since $q(0)=1$ by definition, we thus see that the EGF of the sequence $(q(n))_{n\geq 0}$ is given by $\hat Q(z)=(z+3-2e^z)^{-1}$.

Now, by taking the $(n+1)$st derivative at $z=0$ of both sides of the identity
$$
(z+3-2{\,}e^z){\, }\hat Q(z)~=~1
$$
(using the general Leibniz rule) we immediately derive the claimed recurrence equation for the sequence $(q(n))_{n\geq 0}$.

Let us now establish Eq.~\eqref{eq:Nk}. It is enough to show that the EGF of the sequence $(\tilde{q}(n))_{n\geq 0}$ defined by $\tilde{q}(0)=1$ and
$$
\tilde{q}(n) ~=~ \sum_{i=0}^n 2^i{\,}\sum_{k=0}^{n-i}(-1)^k{\,}{n\choose k}{n-k\brace i}{\,}(i+k)!{\,},\qquad n\geq 1,
$$
is exactly $\hat Q(z)$.

For any integer $i\geq 0$, consider the sequences $(f_n^i)_{n\geq 0}$ and $(g_n^i)_{n\geq 0}$ defined by $f_n^i=(-1)^n(n+i)!$ and $g_n^i={n\brace i}$. Define also the sequence $(h_n^i)_{n\geq 0}$ by the binomial convolution of $(f_n^i)_{n\geq 0}$ and $(g_n^i)_{n\geq 0}$, that is,
$$
h_n^i ~=~ \sum_{k=0}^n{n\choose k}(-1)^k(i+k)!{n-k\brace i}.
$$
Observing that ${n-k\brace i}=0$ if $n-k<i$ we see that
\begin{equation}\label{eq:q2h}
\tilde{q}(n) ~=~ \sum_{i=0}^n 2^ih_n^i{\,},\qquad n\geq 0.
\end{equation}
Let $\hat F_i(z)$, $\hat G_i(z)$, and $\hat H_i(z)$ be the EGFs of the sequences $(f_n^i)_{n\geq 0}$, $(g_n^i)_{n\geq 0}$, and $(h_n^i)_{n\geq 0}$, respectively. It is known (see, e.g., \cite[p.~335, p.~351]{GraKnuPat94}) that $\hat F_i(z)=i!(z+1)^{-i-1}$ and $\hat G_i(z)=(e^z-1)^i/i!$. We then have
$$
\hat H_i(z) ~=~ \hat F_i(z)\hat G_i(z) ~=~ \frac{(e^z-1)^i}{(z+1)^{i+1}}{\,}.
$$
Since $h_n^i=D_z^n \hat H_i(z)|_{z=0}$, using \eqref{eq:q2h} we obtain
$$
\tilde{q}(n) ~=~ D_z^n{\,}\frac{1-\big(2\frac{e^z-1}{z+1}\big)^{n+1}}{z+3-2e^z}\Big|_{z=0} ~=~ D_z^n{\,}\frac{1}{z+3-2e^z}\Big|_{z=0} ~=~ (D_z^n{\,}\hat Q)(0).
$$
This means that the EGF of $(\tilde{q}(n))_{n\geq 0}$ is given by $\hat Q(z)$. This completes the proof.
\end{proof}

\begin{remark}\label{rem:r4zt}
\begin{enumerate}
\item[(a)] It is clear that the radius $r$ of convergence of the series $\hat Q(z)$ is less than or equal to the closest singularity ($\approx 0.583$) to the origin of the real function $x\mapsto 1/(x+3-2e^x)$. We conjecture that $r$ is given by the classical ratio test and corresponds exactly to that singularity. In mathematical terms, this amounts to proving (or disproving) that
$$
\frac{q(n+1)}{(n+1){\,}q(n)} ~\to ~\frac{1}{r} ~\approx ~ 1.715\qquad\text{as $~n\to\infty$},
$$
where $r\approx 0.583$ is the unique positive zero of the real function $x\mapsto x+3-2e^x$.
\item[(b)] In the proof of Theorem~\ref{thm:Nk} we have established Eq.~\eqref{eq:Nk} by first searching for the explicit form of $\hat Q(z)$ from the definition of the sequence $(q(n))_{n\geq 0}$. In the appendix we provide an alternative proof of \eqref{eq:Nk} that does not make use of $\hat Q(z)$.
\end{enumerate}
\end{remark}

\begin{proposition}\label{prop:qe}
For any integer $n\geq 0$, we have $q_e(n)=q_a(n)=n{\,}q(n-1)$ and $q_{ea}(n)=n(n-1){\,}q(n-2)$.
\end{proposition}

\begin{proof}
Let us first show how we can construct an arbitrary associative and quasitrivial operation $F\colon X_n^2\to X_n$ having a neutral element. There are $n$ ways to choose the neutral element $e$ in $X_n$. Then we observe that the restriction of $F$ to $(X_n\setminus\{e\})^2$ is still an associative and quasitrivial operation, so we have $q(n-1)$ possible choices to construct this restriction. This shows that $q_e(n)=n{\,}q(n-1)$. Using the same reasoning, we also obtain $q_a(n)=n{\,}q(n-1)$ and $q_{ea}(n)=n(n-1){\,}q(n-2)$.
\end{proof}

\begin{table}[htbp]
$$
\begin{array}{|c|rrrr|}
\hline n & q(n) & q_e(n) & q_a(n) & q_{ea}(n) \\
\hline 0 & 1 & 0 & 0 & 0 \\
1 & 1 & 1 & 1 & 0 \\
2 & 4 & 2 & 2 & 2 \\
3 & 20 & 12 & 12 & 6 \\
4 & 138 & 80 & 80 & 48 \\
5 & 1{\,}182 & 690 & 690 & 400 \\
6 & 12{\,}166 & 7{\,}092 & 7{\,}092 & 4{\,}140 \\
\hline
\mathrm{OEIS}^{\mathstrut} & \mathrm{A292932} & \mathrm{A292933} & \mathrm{A292933} & \mathrm{A292934} \\
\hline
\end{array}
$$
\caption{First few values of $q(n)$, $q_e(n)$, $q_a(n)$, and $q_{ea}(n)$}
\label{tab:q}
\end{table}

We now consider the subclass of associative and quasitrivial operations $F\colon X_n^2\to X_n$ that are $\leq_n$-preserving. To this extent, we introduce a generalization of single-peakedness to weak orderings, that we call \emph{weak single-peakedness}. This leads to a generalization of Proposition~\ref{prop:Qspl} to arbitrary quasitrivial semigroups (see Proposition~\ref{prop:Qwspl}). We will further elaborate on this concept in Section~\ref{sec:5}.

\begin{definition}\label{de:wBlack}
Let $\leq$ be a total ordering on $X$ and let $\precsim$ be a weak ordering on $X$. We say that $\precsim$ is \emph{weakly single-peaked for $\leq$} if for any $a,b,c\in X$ such that $a<b<c$, we have $b\prec a$ or $b\prec c$ or $a\sim b\sim c$.
\end{definition}

\begin{proposition}\label{prop:Qwspl}
Let $\leq$ be a total ordering on $X$ and let $\precsim$ be a weak ordering on $X$. Suppose that $F\colon X^2\to X$ is of the form \eqref{eq:kimura}. Then $F$ is $\leq$-preserving if and only if $\precsim$ is weakly single-peaked for $\leq$.
\end{proposition}

\begin{proof}
(Necessity) We proceed by contradiction. Suppose that there exist $a,b,c\in X$ satisfying $a<b<c$ such that $a\precsim b$ and $c\precsim b$ and $\neg(a\sim b\sim c)$. Suppose that $a\prec b$ and $c\sim b$. The other two cases can be dealt with similarly.
\begin{itemize}
\item If $F|_{[b]_{\sim}^2}=\pi_1|_{[b]_{\sim}^2}$, then by $\leq$-preservation of $F$ we have
$
b=F(a,b)\leq F(a,c)\leq F(b,c)=b.
$
\item If $F|_{[b]_{\sim}^2}=\pi_2|_{[b]_{\sim}^2}$, then by $\leq$-preservation of $F$ we have
$
b=F(b,a)\leq F(c,a)\leq F(c,b)=b.
$
\end{itemize}
In the first (resp.\ second) case we obtain $F(a,c)=b$ (resp.\ $F(c,a)=b$), which contradicts quasitriviality.

(Sufficiency) We proceed by contradiction. Suppose that $\precsim$ is weakly single-peaked for $\leq$ and that $F$ is not $\leq$-preserving. Then, for instance there exist $x,y,z\in X$ such that
\begin{equation}\label{ndecP}
y<z\quad\text{and}\quad F(x,y)>F(x,z).
\end{equation}

Using \eqref{ndecP} it is easy to see by contradiction that we necessarily have
$$
(x\precsim y\hspace{1ex}\text{or}\hspace{1ex} x\precsim z)\quad\text{and}\quad (y\precsim x\hspace{1ex}\text{or}\hspace{1ex} z\precsim x).
$$
We then have only the following three mutually exclusive cases to consider.
\begin{itemize}
\item If $y\prec x\prec z$ or $y\sim x\prec z$ or $y\prec x\sim z$, then by \eqref{ndecP} we obtain $y<z<x$, which violates weak single-peakedness.
\item If $z\prec x\prec y$ or $z\sim x\prec y$ or $z\prec x\sim y$, then by \eqref{ndecP} we obtain $x<y<z$, which violates weak single-peakedness.
\item If $x\sim y\sim z$, then we must have $F|_{[x]_{\sim}^2}=\pi_1|_{[x]_{\sim}^2}$ or $F|_{[x]_{\sim}^2}=\pi_2|_{[x]_{\sim}^2}$, which immediately violates \eqref{ndecP}.
\end{itemize}
This completes the proof of Proposition~\ref{prop:Qwspl}.
\end{proof}

From Proposition~\ref{prop:Qwspl} we immediately derive the following characterization of the class of associative, quasitrivial, and order-preserving operations $F\colon X^2\to X$, thus generalizing to the noncommutative case the equivalence between assertions (i) and (ii) of Theorem~\ref{thm:cchar1}. We observe that, when $X=X_n$ for some integer $n\geq 1$, an alternative characterisation of this class has been recently presented in \cite{Kis}.

\begin{theorem}\label{thm:QtAsNdKi}
Let $\leq$ be a total ordering on $X$. An $F\colon X^2\to X$ is associative, quasitrivial, and $\leq$-preserving if and only if it is of the form \eqref{eq:kimura} for some weak ordering $\precsim$ on $X$ that is weakly single-peaked for $\leq$.
\end{theorem}

We now consider the problem of enumerating associative and quasitrivial operations $F\colon X_n^2\to X_n$ that are $\leq_n$-preserving. We will make use of the following two auxiliary lemmas.

\begin{lemma}\label{lemma:condSt}
Let $\leq$ be a total ordering on $X$ and let $\precsim$ be a weak ordering on $X$. If $\precsim$ is weakly single-peaked for $\leq$, then there are no pairwise distinct $a,b,c,d\in X$ such that $a\prec b\sim c\sim d$.
\end{lemma}

\begin{proof}
We proceed by contradiction. Suppose that there exist pairwise distinct $a,b,c,d\in X$ such that $a\prec b\sim c\sim d$. Assume without loss of generality that $b<c<d$. If $b<a<c$, then the set $\{a,c,d\}$ violates weak single-peakedness of $\precsim$. In the three other cases the set $\{a,b,c\}$ violates weak single-peakedness of $\precsim$.
\end{proof}

\begin{lemma}\label{lemma:maxEl}
Let $\leq$ be a total ordering on $X$ and let $\precsim$ be a weak ordering on $X$ that is weakly single-peaked for $\leq$. Assume that both $\min_{\leq}X$ and $\max_{\leq}X$ are nonempty and let $a=\min_{\leq}X$ and $b=\max_{\leq}X$. If $\max_{\precsim}X\neq X$, then $\max_{\precsim}X\subseteq\{a,b\}$.
\end{lemma}

\begin{proof}
By Lemma~\ref{lemma:condSt} the set $\max_{\precsim}X$ contains at most two elements. Now suppose that there exists $x\in (\max_{\precsim}X)\setminus\{a,b\}$. Then the set $\{a,x,b\}$ violates weak single-peakedness of $\precsim$.
\end{proof}

Assume that $X_n$ is endowed with $\leq_n$. For any integer $n\geq 0$, we denote by $u(n)$ the number of weak orderings $\precsim$ on $X_n$ that are weakly single-peaked for $\leq_n$. Also, we denote by
\begin{itemize}
\item $u_e(n)$ the number of weak orderings $\precsim$ on $X_n$ that are weakly single-peaked for $\leq_n$ and for which $X_n$ has exactly one minimal element for $\precsim$,
\item $u_a(n)$ the number of weak orderings $\precsim$ on $X_n$ that are weakly single-peaked for $\leq_n$ and for which $X_n$ has exactly one maximal element for $\precsim$,
\item $u_{ea}(n)$ the number of weak orderings $\precsim$ on $X_n$ that are weakly single-peaked for $\leq_n$ and for which $X_n$ has exactly one minimal element and exactly one maximal element for $\precsim$, the two elements being distinct.
\end{itemize}
As a convention, we set $u(0)=u_e(0)=u_a(0)=u_{ea}(0)=0$. Propositions~\ref{prop:vn} and \ref{prop:ven} below provide explicit formulas for these sequences. The first few values of these sequences are shown in Table~\ref{tab:u}.\footnote{Note that the sequences A048739 and A163271 are shifted versions of $(u(n))_{n\geq 0}$ and $(u_{ea}(n))_{n\geq 0}$, respectively. More precisely, we have $u(n)=\mathrm{A048739}(n-1)$ and $u_{ea}(n)=\mathrm{A163271}(n-1)$ for every integer $n\geq 1$.} It turns out that the sequence $(u_e(n))_{n\geq 0}$ consists of the so-called \emph{Pell numbers} (Sloane's A000129).

\begin{proposition}\label{prop:vn}
The sequence $(u(n))_{n\geq 0}$ satisfies the second order linear recurrence equation
$$
u(n+2)-2{\,}u(n+1)-u(n) ~=~ 1{\,},\qquad n\geq 0,
$$
with $u(0)=0$ and $u(1)=1$, and we have
\begin{eqnarray*}
2{\,}u(n)+1 &=& \textstyle{\frac{1}{2}(1+\sqrt{2})^{n+1}+\frac{1}{2}(1-\sqrt{2})^{n+1}}\\
&=& \textstyle{\sum_{k\geq 0}{n+1\choose 2k}{\,}2^k}{\,},\qquad n\geq 0.
\end{eqnarray*}
Moreover, its GF is given by $U(z)=z/(z^3+z^2-3z+1)$.
\end{proposition}

\begin{proof}
We clearly have $u(0)=0$ and $u(1)=1$. So let us assume that $n\geq 2$. If $\precsim$ is a weak ordering on $X_n$ that is weakly single-peaked for $\leq_n$, then by Lemma~\ref{lemma:maxEl} either $\max_{\precsim}X_n=X_n$, or $\max_{\precsim}X_n=\{1\}$, or $\max_{\precsim}X_n=\{n\}$, or $\max_{\precsim}X_n=\{1,n\}$. In the three latter cases it is clear that the restriction of $\precsim$ to $X_n\setminus\max_{\precsim}X_n$ is weakly single-peaked for the restriction of $\leq_n$ to $X_n\setminus\max_{\precsim}X_n$. It follows that the number $u(n)$ of weakly single-peaked weak orderings on $X_n$ for $\leq_n$ satisfies the following second order linear equation
$$
u(n) ~=~ 1 + u(n-1)+u(n-1)+u(n-2),\qquad n\geq 2.
$$
The claimed expressions of $u(n)$ and GF of $(u(n))_{n\geq 0}$ follow straightforwardly.
\end{proof}

\begin{proposition}\label{prop:ven}
The sequence $(u_e(n))_{n\geq 0}$ satisfies the second order linear recurrence equation
$$
u_e(n+2)-2{\,}u_e(n+1)-u_e(n) ~=~ 0{\,},\qquad n\geq 0,
$$
with $u_e(0)=0$ and $u_e(1)=1$, and we have
\begin{eqnarray*}
u_e(n) &=& \textstyle{\frac{\sqrt{2}}{4}(1+\sqrt{2})^n-\frac{\sqrt{2}}{4}(1-\sqrt{2})^n}\\
&=& \textstyle{\sum_{k\geq 0}{n\choose 2k+1}{\,}2^k}{\,},\qquad n\geq 0.
\end{eqnarray*}
Moreover, its GF is given by $U_e(z)=-z/(z^2+2z-1)$. Furthermore, for any integer $n\geq 1$, we have $u_a(n)=2u(n-1)$, $u_{ea}(n)=2u_e(n-1)$, and $u_a(0)=u_{ea}(0)=0$.
\end{proposition}

\begin{proof}
The formula describing the sequence $(u_e(n))_{n\geq 0}$ is obtained by following the same steps as in the proof of Proposition~\ref{prop:vn}, except that in this case we always have $\max_{\precsim}X_n\neq X_n$. As for the sequence $(u_a(n))_{n\geq 0}$ we note that $\max_{\precsim}X_n$ must be either $\{1\}$ or $\{n\}$ and that the restriction of $\precsim$ to $X_n\setminus\max_{\precsim}X_n$ is weakly single-peaked for the restriction of $\leq_n$ to $X_n\setminus\max_{\precsim}X_n$. We proceed similarly for the sequence $(u_{ea}(n))_{n\geq 0}$.
\end{proof}

\begin{table}[htbp]
$$
\begin{array}{|c|rrrr|}
\hline n & u(n) & u_e(n) & u_a(n) & u_{ea}(n) \\
\hline 0 & 0 & 0 & 0 & 0 \\
1 & 1 & 1 & 0 & 0 \\
2 & 3 & 2 & 2 & 2 \\
3 & 8 & 5 & 6 & 4 \\
4 & 20 & 12 & 16 & 10 \\
5 & 49 & 29 & 40 & 24 \\
6 & 119 & 70 & 98 & 58 \\
\hline
\mathrm{OEIS}^{\mathstrut} & \mathrm{A048739} & \mathrm{A000129} & \mathrm{A293004} & \mathrm{A163271} \\
\hline
\end{array}
$$
\caption{First few values of $u(n)$, $u_e(n)$, $u_a(n)$, and $u_{ea}(n)$}
\label{tab:u}
\end{table}

\begin{example}
The $u(3)=8$ weak orderings on $X_3$ that are weakly single-peaked for $\leq_3$ are: $1\prec 2\prec 3$, $2\prec 1\prec 3$, $2\prec 3\prec 1$, $3\prec 2\prec 1$, $2\prec 1\sim 3$, $1\sim 2\prec 3$, $2\sim 3\prec 1$, and $1\sim 2\sim 3$. $u_e(3)=5$ of those have exactly one minimal element and $u_a(3)=6$ of those have exactly one maximal element. $u_{ea}(3)=4$ of those have exactly one minimal element and exactly one maximal element. These four weak orderings correspond to the $2^{3-1}=4$ total orderings on $X_3$ that are single-peaked for $\leq_3$.
\end{example}

Assume again that $X_n$ is endowed with $\leq_n$. For any integer $n\geq 0$, we denote by $v(n)$ the number of associative, quasitrivial, and $\leq_n$-preserving operations $F\colon X_n^2\to X_n$. Also, we denote by
\begin{itemize}
\item $v_e(n)$ the number of associative, quasitrivial, and $\leq_n$-preserving operations $F\colon X_n^2\to X_n$ that have neutral elements,
\item $v_a(n)$ the number of associative, quasitrivial, and $\leq_n$-preserving operations $F\colon X_n^2\to X_n$ that have annihilator elements,
\item $v_{ea}(n)$ the number of associative, quasitrivial, and $\leq_n$-preserving operations $F\colon X_n^2\to X_n$ that have distinct neutral and annihilator elements.
\end{itemize}
As a convention, we set $v(0)=v_e(0)=v_a(0)=v_{ea}(0)=0$. Propositions~\ref{prop:rn} and \ref{prop:ren} below provide explicit formulas for these sequences. The first few values of these sequences are shown in Table~\ref{tab:v}.

\begin{proposition}\label{prop:rn}
The sequence $(v(n))_{n\geq 0}$ satisfies the second order linear recurrence equation
$$
v(n+2)-2{\,}v(n+1)-2{\,}v(n) ~=~ 2{\,},\qquad n\geq 0,
$$
with $v(0)=0$ and $v(1)=1$, and we have
\begin{eqnarray*}
3{\,}v(n)+2 &=& \textstyle{\frac{2+\sqrt{3}}{2}(1+\sqrt{3})^n+\frac{2-\sqrt{3}}{2}(1-\sqrt{3})^n}\\
&=& \textstyle{\sum_{k\geq 0}3^k(2{n\choose 2k}+3{n\choose 2k+1})}{\,},\qquad n\geq 0.
\end{eqnarray*}
Moreover, its GF is given by $V(z)=z(z+1)/(2z^3-3z+1)$.
\end{proposition}

\begin{proof}
We clearly have $v(0)=0$ and $v(1)=1$. So let us assume that $n\geq 2$. If $F\colon X_n^2\to X_n$ is an associative, quasitrivial, and $\leq_n$-preserving operation, then by Theorem~\ref{thm:QtAsNdKi} it is of the form \eqref{eq:kimura} for some weak ordering $\precsim$ on $X_n$ that is weakly single-peaked for $\leq_n$. By Lemma~\ref{lemma:maxEl}, either $\max_{\precsim}X_n=X_n$ or $\max_{\precsim}X_n=\{1\}$ or $\max_{\precsim}X_n=\{n\}$ or $\max_{\precsim}X_n=\{1,n\}$. In the first case we have to consider the two projections $F=\pi_1$ and $F=\pi_2$. In the three latter cases it is clear that the restriction of $F$ to $(X_n\setminus\max_{\precsim}X_n)^2$ is associative, quasitrivial, and $\leq'_n$-preserving, where $\leq'_n$ is the restriction of $\leq_n$ to $X_n\setminus\max_{\precsim}X_n$. Also, in the latter case we have to consider the two projections $F|_{\{1,n\}^2}=\pi_1|_{\{1,n\}^2}$ and $F|_{\{1,n\}^2}=\pi_2|_{\{1,n\}^2}$. It follows that the number $v(n)$ of associative, quasitrivial, and $\leq_n$-preserving operations $F\colon X_n^2\to X_n$ satisfies the following second order linear equation
$$
v(n) ~=~ 2 + v(n-1)+v(n-1)+2v(n-2),\qquad n\geq 2.
$$
The claimed expressions of $v(n)$ and GF of $(v(n))_{n\geq 0}$ follow straightforwardly.
\end{proof}

\begin{proposition}\label{prop:ren}
The sequence $(v_e(n))_{n\geq 0}$ satisfies the second order linear recurrence equation
$$
v_e(n+2)-2{\,}v_e(n+1)-2v_e(n) ~=~ 0{\,},\qquad n\geq 0,
$$
with $v_e(0)=0$ and $v_e(1)=1$, and we have
$$
\textstyle{v_e(n) ~=~ \frac{\sqrt{3}}{6}(1+\sqrt{3})^n-\frac{\sqrt{3}}{6}(1-\sqrt{3})^n ~=~ \sum_{k\geq 0}{n\choose 2k+1}{\,}3^k{\,},\qquad n\geq 0.}
$$
Moreover, its GF is given by $V_e(z)=-z/(2z^2+2z-1)$. Furthermore, for any integer $n\geq 1$, we have $v_a(n)=2v(n-1)$, $v_{ea}(n)=2v_e(n-1)$, and $v_a(0)=v_{ea}(0)=0$.
\end{proposition}

\begin{proof}
The formula describing the sequence $(v_e(n))_{n\geq 0}$ is obtained by following the same steps as in the proof of Proposition~\ref{prop:rn}, except that in this case we always have $\max_{\precsim}X_n\neq X_n$. As for the sequence $(v_a(n))_{n\geq 0}$ we note that $\max_{\precsim}X_n$ must be either $\{1\}$ or $\{n\}$ and that the restriction of $F$ to $(X_n\setminus\max_{\precsim}X_n)^2$ is associative, quasitrivial, and $\leq'_n$-preserving, where $\leq'_n$ is the restriction of $\leq_n$ to $X_n\setminus\max_{\precsim}X_n$. We proceed similarly for the sequence $(v_{ea}(n))_{n\geq 0}$.
\end{proof}

\begin{table}[htbp]
$$
\begin{array}{|c|rrrr|}
\hline n & v(n) & v_e(n) & v_a(n) & v_{ea}(n) \\
\hline 0 & 0 & 0 & 0 & 0 \\
1 & 1 & 1 & 0 & 0  \\
2 & 4 & 2 & 2 & 2 \\
3 & 12 & 6 & 8 & 4 \\
4 & 34 & 16 & 24 & 12 \\
5 & 94 & 44 & 68 & 32 \\
6 & 258 & 120 & 188 & 88 \\
\hline
\mathrm{OEIS}^{\mathstrut} & \mathrm{A293005} & \mathrm{A002605} & \mathrm{A293006} & \mathrm{A293007} \\
\hline
\end{array}
$$
\caption{First few values of $v(n)$, $v_e(n)$, $v_a(n)$, and $v_{ea}(n)$}
\label{tab:v}
\end{table}

\begin{example}
We show in Figure~\ref{fig:eos4} the $q(3)=20$ associative and quasitrivial operations on $X_3$. Among these operations, $q_e(3)=12$ have neutral elements and $v(3)=12$ are $\leq_3$-preserving.
\end{example}

\setlength{\unitlength}{3.5ex}
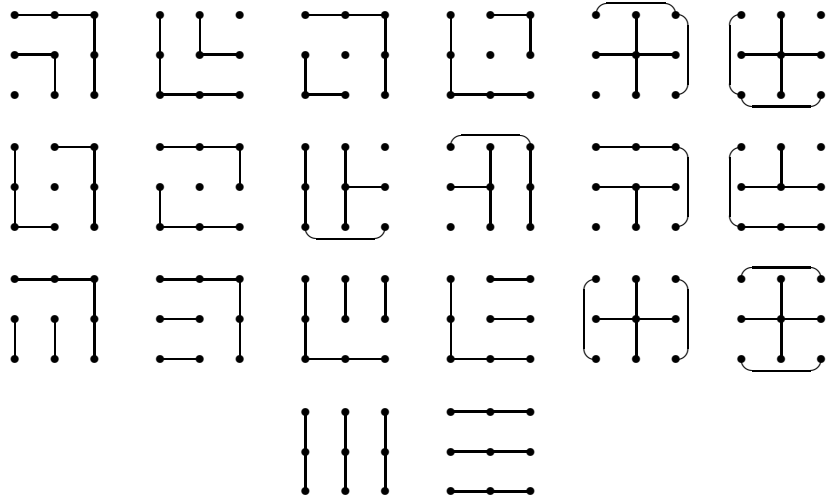
\begin{figure}[htbp]
\begin{center}
\begin{small}
\begin{picture}(3,3)
\multiput(0.5,0.5)(0,1){3}{\multiput(0,0)(1,0){3}{\circle*{0.2}}}
\drawline[1](0.5,2.5)(2.5,2.5)(2.5,0.5)\drawline[1](0.5,1.5)(1.5,1.5)(1.5,0.5)
\end{picture}
\hspace{0.01\textwidth}
\begin{picture}(3,3)
\multiput(0.5,0.5)(0,1){3}{\multiput(0,0)(1,0){3}{\circle*{0.2}}}
\drawline[1](0.5,2.5)(0.5,0.5)(2.5,0.5)\drawline[1](1.5,2.5)(1.5,1.5)(2.5,1.5)
\end{picture}
\hspace{0.01\textwidth}
\begin{picture}(3,3)
\multiput(0.5,0.5)(0,1){3}{\multiput(0,0)(1,0){3}{\circle*{0.2}}}
\drawline[1](0.5,2.5)(2.5,2.5)(2.5,0.5)\drawline[1](0.5,1.5)(0.5,0.5)(1.5,0.5)
\end{picture}
\hspace{0.01\textwidth}
\begin{picture}(3,3)
\multiput(0.5,0.5)(0,1){3}{\multiput(0,0)(1,0){3}{\circle*{0.2}}}
\drawline[1](0.5,2.5)(0.5,0.5)(2.5,0.5)\drawline[1](1.5,2.5)(2.5,2.5)(2.5,1.5)
\end{picture}
\hspace{0.01\textwidth}
\begin{picture}(3,3)
\multiput(0.5,0.5)(0,1){3}{\multiput(0,0)(1,0){3}{\circle*{0.2}}}
\drawline[1](0.5,1.5)(2.5,1.5)\drawline[1](1.5,0.5)(1.5,2.5)
\put(2.5,1.5){\oval(0.6,2)[r]}\put(1.5,2.5){\oval(2,0.6)[t]}
\end{picture}
\hspace{0.01\textwidth}
\begin{picture}(3,3)
\multiput(0.5,0.5)(0,1){3}{\multiput(0,0)(1,0){3}{\circle*{0.2}}}
\drawline[1](0.5,1.5)(2.5,1.5)\drawline[1](1.5,0.5)(1.5,2.5)
\put(0.5,1.5){\oval(0.6,2)[l]}\put(1.5,0.5){\oval(2,0.6)[b]}
\end{picture}

\medskip

\begin{picture}(3,3)
\multiput(0.5,0.5)(0,1){3}{\multiput(0,0)(1,0){3}{\circle*{0.2}}}
\drawline[1](0.5,2.5)(0.5,0.5)(1.5,0.5)\drawline[1](1.5,2.5)(2.5,2.5)(2.5,0.5)
\end{picture}
\hspace{0.01\textwidth}
\begin{picture}(3,3)
\multiput(0.5,0.5)(0,1){3}{\multiput(0,0)(1,0){3}{\circle*{0.2}}}
\drawline[1](0.5,2.5)(2.5,2.5)(2.5,1.5)\drawline[1](0.5,1.5)(0.5,0.5)(2.5,0.5)
\end{picture}
\hspace{0.01\textwidth}
\begin{picture}(3,3)
\multiput(0.5,0.5)(0,1){3}{\multiput(0,0)(1,0){3}{\circle*{0.2}}}
\drawline[1](0.5,2.5)(0.5,0.5)\drawline[1](1.5,2.5)(1.5,0.5)\drawline[1](1.5,1.5)(2.5,1.5)
\put(1.5,0.5){\oval(2,0.6)[b]}
\end{picture}
\hspace{0.01\textwidth}
\begin{picture}(3,3)
\multiput(0.5,0.5)(0,1){3}{\multiput(0,0)(1,0){3}{\circle*{0.2}}}
\drawline[1](2.5,2.5)(2.5,0.5)\drawline[1](1.5,2.5)(1.5,0.5)\drawline[1](0.5,1.5)(1.5,1.5)
\put(1.5,2.5){\oval(2,0.6)[t]}
\end{picture}
\hspace{0.01\textwidth}
\begin{picture}(3,3)
\multiput(0.5,0.5)(0,1){3}{\multiput(0,0)(1,0){3}{\circle*{0.2}}}
\drawline[1](0.5,2.5)(2.5,2.5)\drawline[1](0.5,1.5)(2.5,1.5)\drawline[1](1.5,1.5)(1.5,0.5)
\put(2.5,1.5){\oval(0.6,2)[r]}
\end{picture}
\hspace{0.01\textwidth}
\begin{picture}(3,3)
\multiput(0.5,0.5)(0,1){3}{\multiput(0,0)(1,0){3}{\circle*{0.2}}}
\drawline[1](0.5,0.5)(2.5,0.5)\drawline[1](0.5,1.5)(2.5,1.5)\drawline[1](1.5,1.5)(1.5,2.5)
\put(0.5,1.5){\oval(0.6,2)[l]}
\end{picture}

\medskip

\begin{picture}(3,3)
\multiput(0.5,0.5)(0,1){3}{\multiput(0,0)(1,0){3}{\circle*{0.2}}}
\drawline[1](0.5,2.5)(2.5,2.5)(2.5,0.5)\drawline[1](0.5,1.5)(0.5,0.5)\drawline[1](1.5,1.5)(1.5,0.5)
\end{picture}
\hspace{0.01\textwidth}
\begin{picture}(3,3)
\multiput(0.5,0.5)(0,1){3}{\multiput(0,0)(1,0){3}{\circle*{0.2}}}
\drawline[1](0.5,2.5)(2.5,2.5)(2.5,0.5)\drawline[1](0.5,1.5)(1.5,1.5)\drawline[1](0.5,0.5)(1.5,0.5)
\end{picture}
\hspace{0.01\textwidth}
\begin{picture}(3,3)
\multiput(0.5,0.5)(0,1){3}{\multiput(0,0)(1,0){3}{\circle*{0.2}}}
\drawline[1](0.5,2.5)(0.5,0.5)(2.5,0.5)\drawline[1](1.5,2.5)(1.5,1.5)\drawline[1](2.5,2.5)(2.5,1.5)
\end{picture}
\hspace{0.01\textwidth}
\begin{picture}(3,3)
\multiput(0.5,0.5)(0,1){3}{\multiput(0,0)(1,0){3}{\circle*{0.2}}}
\drawline[1](0.5,2.5)(0.5,0.5)(2.5,0.5)\drawline[1](1.5,2.5)(2.5,2.5)\drawline[1](1.5,1.5)(2.5,1.5)
\end{picture}
\hspace{0.01\textwidth}
\begin{picture}(3,3)
\multiput(0.5,0.5)(0,1){3}{\multiput(0,0)(1,0){3}{\circle*{0.2}}}
\drawline[1](0.5,1.5)(2.5,1.5)\drawline[1](1.5,0.5)(1.5,2.5)
\put(2.5,1.5){\oval(0.6,2)[r]}\put(0.5,1.5){\oval(0.6,2)[l]}
\end{picture}
\hspace{0.01\textwidth}
\begin{picture}(3,3)
\multiput(0.5,0.5)(0,1){3}{\multiput(0,0)(1,0){3}{\circle*{0.2}}}
\drawline[1](0.5,1.5)(2.5,1.5)\drawline[1](1.5,0.5)(1.5,2.5)
\put(1.5,0.5){\oval(2,0.6)[b]}\put(1.5,2.5){\oval(2,0.6)[t]}
\end{picture}

\medskip

\begin{picture}(3,3)
\multiput(0.5,0.5)(0,1){3}{\multiput(0,0)(1,0){3}{\circle*{0.2}}}
\drawline[1](0.5,0.5)(0.5,2.5)\drawline[1](1.5,0.5)(1.5,2.5)\drawline[1](2.5,0.5)(2.5,2.5)
\end{picture}
\hspace{0.01\textwidth}
\begin{picture}(3,3)
\multiput(0.5,0.5)(0,1){3}{\multiput(0,0)(1,0){3}{\circle*{0.2}}}
\drawline[1](0.5,2.5)(2.5,2.5)\drawline[1](0.5,1.5)(2.5,1.5)\drawline[1](0.5,0.5)(2.5,0.5)
\end{picture}

\end{small}
\caption{The 20 associative and quasitrivial operations on $X_3$}
\label{fig:eos4}
\end{center}
\end{figure}

\begin{remark}
We observe that the explicit expressions of $v(n)$ and $v_e(n)$ as stated in Propositions~\ref{prop:rn} and \ref{prop:ren} were recently and independently obtained in \cite{Kis} by means of a totally different approach.
\end{remark}

\section{Single-peakedness and weak single-peakedness}
\label{sec:5}

In this section we further analyze the single-peakedness and weak single-peaked{\-}ness properties. In particular, we show how these properties can be easily checked graphically.

Define the \emph{strict convex hull} of $x,y\in X$ for a total ordering $\leq$ on $X$ by $\cleq(x,y)=\{z\in X:x<z<y\}$, if $x<y$, and $\cleq(x,y)=\{z\in X:y<z<x\}$, if $y<x$. Using this concept we can rewrite the definitions of single-peakedness and weak single-peakedness in a more symmetric way.

Accordingly, a total ordering $\preceq$ on $X$ is single-peaked for a (reference) total ordering $\leq$ on $X$ if and only if for any $a,b,c\in X$ such that $b\in\cleq(a,c)$, we have $b\prec a$ or $b\prec c$ (see Definition~\ref{de:Black}). In other words, the condition says that from among three pairwise distinct elements of $X$, the centrist one for $\leq$ is never ranked last by $\preceq$.

A noteworthy characterization of single-peakedness is that for any total orderings $\leq$ and $\preceq$ on $X$, the operation $F=\max_{\preceq}$ is $\leq$-preserving if and only if $\preceq$ is single-peaked for $\leq$ (cf.\ Proposition~\ref{prop:Qspl}).

\begin{remark}\label{rem:dualsp}
It is natural to define the dual version of single-peakedness by saying that from among three pairwise distinct elements of $X$, the centrist one for $\leq$ is never ranked \emph{first} by $\preceq$. By doing so, it is clear that $\preceq$ is single-peaked for $\leq$ if and only if the inverse ordering ${\preceq^{-1}}$ (defined by ${a\preceq^{-1}}b\Leftrightarrow b\preceq a$) is dual single-peaked for $\leq$. For instance, we could replace $\max_{\preceq}$ with $\min_{\preceq}$ and ``single-peaked'' with ``dual single-peaked'' in Proposition~\ref{prop:Qspl}. Thus, considering the single-peakedness property or its dual version is simply a matter of convention.
\end{remark}

The following proposition provides an alternative characterization of single-peak{\-}edness. Recall first that, for any total ordering $\leq$ on $X$, a subset $C$ of $X$ is said to be \emph{convex for $\leq$} if for any $a,b,c\in X$ such that $b\in\cleq(a,c)$, we have that $a,c\in C$ implies $b\in C$.

\begin{proposition}[{see \cite[Proposition 3.10]{DevKisMar}}]\label{prop:spConv}
Let $\leq$ and $\preceq$ be total orderings on $X$. Then $\preceq$ is single-peaked for $\leq$ if and only if for every $t\in X$ the set $\{x\in X: x\preceq t\}$ is convex for $\leq$.
\end{proposition}

The single-peakedness property of a total ordering $\preceq$ on $X$ for some total ordering $\leq$ can often be easily checked (especially if $X$ is finite) by plotting a function, say $f_{\preceq}$, in a rectangular coordinate system in the following way. Represent the reference totally ordered set $(X,\leq)$ on the horizontal axis and the reversed version of the totally ordered set $(X,\preceq)$, that is $(X,{\preceq^{-1}})$, on the vertical axis. The function $f_{\preceq}$ is defined by its graph $\{(x,x):x\in X\}$.\footnote{When $X=X_n$ for some integer $n\geq 1$, the graphical representation of $f_{\preceq}$ is then obtained by joining the points $(1,1),\ldots,(n,n)$ by line segments.} We then see that the total ordering $\preceq$ is single-peaked for $\leq$ if and only if the graph of $f_{\preceq}$ is ``V-free'' in the sense that we cannot find three points $(i,i)$, $(j,j)$, $(k,k)$ in V-shape. Equivalently, $f_{\preceq}$ has only one local maximum.

\begin{example}\label{ex:2fl}
Figure~\ref{fig:spnsp} gives the functions $f_{\preceq}$ and $f_{\preceq'}$ corresponding to the total orderings $4\prec 3\prec 5\prec 2\prec 1\prec 6$ (from Remark~\ref{rem:3ab}(a)) and $6\prec' 5\prec' 2\prec' 1\prec' 3\prec' 4$, respectively, on $X_6$. We see that $\preceq$ is single-peaked for $\leq_6$ since $f_{\preceq}$ has only one local maximum while $\preceq'$ is not single-peaked for $\leq_6$ since $f_{\preceq'}$ has two local maxima (also, the points $(1,1)$, $(3,3)$, $(5,5)$ for instance are in V-shape).
\end{example}

\setlength{\unitlength}{3.5ex}
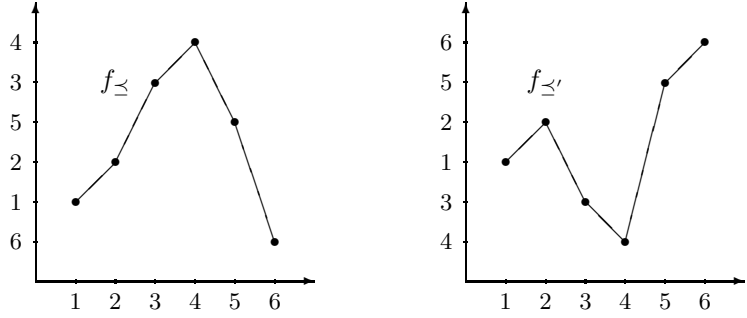
\begin{figure}[htbp]
\begin{center}
\begin{small}
\null\hspace{0.03\textwidth}
\begin{picture}(8,8)
\put(0.5,0.5){\vector(1,0){7}}\put(0.5,0.5){\vector(0,1){7}}
\multiput(1.5,0.45)(1,0){6}{\line(0,1){0.1}}%
\multiput(0.45,1.5)(0,1){6}{\line(1,0){0.1}}%
\put(1.5,0){\makebox(0,0){$1$}}\put(2.5,0){\makebox(0,0){$2$}}\put(3.5,0){\makebox(0,0){$3$}}
\put(4.5,0){\makebox(0,0){$4$}}\put(5.5,0){\makebox(0,0){$5$}}\put(6.5,0){\makebox(0,0){$6$}}
\put(0,1.5){\makebox(0,0){$6$}}\put(0,2.5){\makebox(0,0){$1$}}\put(0,3.5){\makebox(0,0){$2$}}
\put(0,4.5){\makebox(0,0){$5$}}\put(0,5.5){\makebox(0,0){$3$}}\put(0,6.5){\makebox(0,0){$4$}}
\drawline[1](1.5,2.5)(2.5,3.5)(3.5,5.5)(4.5,6.5)(5.5,4.5)(6.5,1.5)
\put(1.5,2.5){\circle*{0.2}}\put(2.5,3.5){\circle*{0.2}}\put(3.5,5.5){\circle*{0.2}}
\put(4.5,6.5){\circle*{0.2}}\put(5.5,4.5){\circle*{0.2}}\put(6.5,1.5){\circle*{0.2}}
\put(2.5,5.5){\makebox(0,0){\begin{normalsize}$f_{\preceq}$\end{normalsize}}}
\end{picture}
\hspace{0.1\textwidth}
\begin{picture}(8,8)
\put(0.5,0.5){\vector(1,0){7}}\put(0.5,0.5){\vector(0,1){7}}
\multiput(1.5,0.45)(1,0){6}{\line(0,1){0.1}}%
\multiput(0.45,1.5)(0,1){6}{\line(1,0){0.1}}%
\put(1.5,0){\makebox(0,0){$1$}}\put(2.5,0){\makebox(0,0){$2$}}\put(3.5,0){\makebox(0,0){$3$}}
\put(4.5,0){\makebox(0,0){$4$}}\put(5.5,0){\makebox(0,0){$5$}}\put(6.5,0){\makebox(0,0){$6$}}
\put(0,1.5){\makebox(0,0){$4$}}\put(0,2.5){\makebox(0,0){$3$}}\put(0,3.5){\makebox(0,0){$1$}}
\put(0,4.5){\makebox(0,0){$2$}}\put(0,5.5){\makebox(0,0){$5$}}\put(0,6.5){\makebox(0,0){$6$}}
\drawline[1](1.5,3.5)(2.5,4.5)(3.5,2.5)(4.5,1.5)(5.5,5.5)(6.5,6.5)
\put(1.5,3.5){\circle*{0.2}}\put(2.5,4.5){\circle*{0.2}}\put(3.5,2.5){\circle*{0.2}}
\put(4.5,1.5){\circle*{0.2}}\put(5.5,5.5){\circle*{0.2}}\put(6.5,6.5){\circle*{0.2}}
\put(2.5,5.5){\makebox(0,0){\begin{normalsize}$f_{\preceq'}$\end{normalsize}}}
\end{picture}
\end{small}
\caption{$\preceq$ is single-peaked (left) while $\preceq'$ is not (right)}
\label{fig:spnsp}
\end{center}
\end{figure}

It is known (see, e.g., \cite{BerPer06}) that there are exactly $2^{n-1}$ single-peaked total orderings on $X_n$ for $\leq_n$. The proof is a simplified version of that of Proposition~\ref{prop:vn} (just observe that either $\max_{\preceq}X_n=\{1\}$ or $\max_{\preceq}X_n=\{n\}$).

Let us now focus on weak single-peakedness. Recall (cf.\ Definition~\ref{de:wBlack}) that a weak ordering $\precsim$ on $X$ is weakly single-peaked for a reference total ordering $\leq$ on $X$ if for any $a,b,c\in X$ such that $b\in\cleq(a,c)$, we have $b\prec a$ or $b\prec c$ or $a\sim b\sim c$.

In Proposition~\ref{prop:Qwspl} we saw that for any total ordering $\leq$ and weak ordering $\precsim$ on $X$, any operation $F\colon X^2\to X$ of the form \eqref{eq:kimura} is $\leq$-preserving if and only if $\precsim$ is weakly single-peaked for $\leq$. This characterization justifies the definition of weak single-peakedness and shows in particular that the condition $a\sim b\sim c$ is necessary in the definition. It is also noteworthy that the following equivalence holds
$$
\text{$b\prec a~$ or $~b\prec c~$ or $~a\sim b\sim c$}\quad\Leftrightarrow\quad
\begin{cases}
a\prec b~\Rightarrow ~ b\prec c,\\
c\prec b~\Rightarrow ~ b\prec a.
\end{cases}
$$

We also have the following alternative characterization of weak single-peakedness. We omit the proof for it is straightforward (by contradiction).

\begin{proposition}\label{prop:52z}
Let $\leq$ be a total ordering on $X$ and let $\precsim$ be a weak ordering on $X$. Then $\precsim$ is weakly single-peaked for $\leq$ if and only if the following conditions hold.
\begin{enumerate}
\item[(a)] For any $a,b,c\in X$ such that $b\in\mathrm{conv}_{\leq}(a,c)$, we have $b\precsim a$ or $b\precsim c$.
\item[(b)] For any $a,b,c\in X$ such that $a\neq c$ and $b\prec a\sim c$, we have $b\in\mathrm{conv}_{\leq}(a,c)$.
\end{enumerate}
\end{proposition}

Weak single-peakedness of a weak ordering $\precsim$ on $X$ for some total ordering $\leq$ can often be visualized and checked by plotting a function $f_{\precsim}$ in a rectangular coordinate system. Represent the reference totally ordered set $(X,\leq)$ on the horizontal axis and the reversed version of the weakly ordered set $(X,\precsim)$ on the vertical axis.\footnote{In this representation, two equivalent elements of $X$ have the same position on the vertical axis; see, e.g., Figures~\ref{fig:2134t} and \ref{fig:1423t}.} Here again the function $f_{\preceq}$ is defined by its graph $\{(x,x):x\in X\}$. Condition (a) of Proposition~\ref{prop:52z} says that the graph of $f_{\precsim}$ is V-free, i.e., we cannot find three points $(i,i)$, $(j,j)$, $(k,k)$ in V-shape. Condition (b) is a little less immediate to interpret graphically. However, Proposition~\ref{prop:53z} below shows how conditions (a) and (b) together can be easily interpreted.

\begin{definition}
Let $\leq$ be a total ordering on $X$ and let $\precsim$ be a weak ordering on $X$. We say that a subset $P$ of $X$ of size $|P|\geq 2$ is \emph{a plateau for $(\leq,\precsim)$} if $P$ is convex for $\leq$ and if there exists $x\in X$ such that $P\subseteq [x]_{\sim}$.
\end{definition}

\begin{proposition}\label{prop:53z}
Let $\leq$ be a total ordering on $X$ and let $\precsim$ be a weak ordering on $X$. Consider the assertions (a) and (b) of Proposition~\ref{prop:52z} as well as the following one.
\begin{enumerate}
\item[(b')] If $P\subseteq X$, $|P|\geq 2$, is a plateau for $(\leq,\precsim)$, then it is $\precsim$-minimal in the sense that for every $a\in X$ satisfying $a\precsim P$ there exists $z\in P$ such that $z\sim a$.
\end{enumerate}
Then we have ((a) and (b')) $\Rightarrow$ (b) and (b) $\Rightarrow$ (b').
\end{proposition}

\begin{proof}
Let us prove that ((a) and (b')) implies (b). Let $a,b,c\in X$ such that $a\neq c$ and $b\prec a\sim c$ and suppose that $b\notin\cleq(a,c)$. Assume without loss of generality that $b<a$. If $\cleq(a,c)$ is a plateau for $(\leq,\precsim)$, then it cannot be $\precsim$-minimal, which contradicts (b'). Hence $\cleq(a,c)$ is not a plateau for $(\leq,\precsim)$, which means that there exists $z\in\cleq(a,c)$ such that $\neg(z\sim a)$. By (a) we then have $z\prec a$. But then the set $\{a,b,z\}$ violates condition (a) since the points $(b,b)$, $(a,a)$, $(z,z)$ are in V-shape.

Let us now prove that (b) implies (b'). Let $P\subseteq X$, $|P|\geq 2$, be a plateau for $(\leq,\precsim)$ and let $a,c\in P$, $a\neq c$. Suppose that $P$ is not $\precsim$-minimal, i.e., there exists $b\in X$ such that $b\prec a\sim c$. By (b), we have $b\in\mathrm{conv}_{\leq}(a,c)$, which contradicts the fact that $P$ is a plateau for $(\leq,\precsim)$.
\end{proof}

From Proposition~\ref{prop:53z} it follows that conditions (a) and (b) hold if and only if conditions (a) and (b') hold. As discussed above, condition (a) says that the graph of $f_{\precsim}$ is V-free. Now, condition (b') simply says that the graph of $f_{\precsim}$ is both reversed L-free and L-free, which means that the two patterns shown in Figure~\ref{fig:fpat} (reversed L-shape and L-shape), where each horizontal part is a plateau $P$, are forbidden.

\setlength{\unitlength}{4ex}
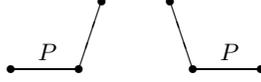
\begin{figure}[htbp]
\begin{center}
\begin{small}
\begin{picture}(6.5,2.5)
\drawline[1](0.5,0.5)(2,0.5)(2.5,2)\drawline[1](4,2)(4.5,0.5)(6,0.5)
\put(0.5,0.5){\circle*{0.15}}\put(2,0.5){\circle*{0.15}}\put(2.5,2){\circle*{0.15}}
\put(4,2){\circle*{0.15}}\put(4.5,0.5){\circle*{0.15}}\put(6,0.5){\circle*{0.15}}
\put(1.1,0.7){$P$}\put(5.1,0.7){$P$}
\end{picture}
\end{small}
\caption{The two patterns excluded by condition (b')}
\label{fig:fpat}
\end{center}
\end{figure}

Summing up, we have proved the following result.

\begin{theorem}\label{thm:VLL5}
Let $\leq$ be a total ordering on $X$ and let $\precsim$ be a weak ordering on $X$. Then $\precsim$ is weakly single-peaked for $\leq$ if and only if conditions (a) and (b') of Propositions~\ref{prop:52z} and \ref{prop:53z} hold (i.e., the graph of $f_{\precsim}$ is V-free, L-free, and reversed L-free).
\end{theorem}

\begin{example}\label{ex:2134t}
Let us consider the operation $F\colon X_4^2\to X_4$ shown in Figure~\ref{fig:2134t} (left). Using the tests given at the end of Section~2 for instance, we can see that this operation is associative and quasitrivial. It is also $\leq_4$-preserving and such that $\deg_F=(0,3,3,6)$. Thus, $F$ is of the form \eqref{eq:kimura}, where $\precsim$ is the weak ordering on $X_4$ obtained by ranking the numbers $\deg_F(x)$, $x\in X_4$, in nondecreasing order, that is, $2\prec 1\sim 3\prec 4$; see Figure~\ref{fig:2134t} (center). By Proposition~\ref{prop:Qwspl} this weak ordering $\precsim$ is weakly single-peaked for $\leq_4$. By Theorem~\ref{thm:VLL5} the graph of $f_{\precsim}$ is V-free, L-free, and reversed L-free; see Figure~\ref{fig:2134t} (right).
\end{example}

\setlength{\unitlength}{4.5ex}
\begin{figure}[htbp]
\begin{center}
\begin{small}
\begin{picture}(4,4)
\multiput(0.5,0.5)(0,1){4}{\multiput(0,0)(1,0){4}{\circle*{0.18}}}
\drawline[1](0.5,3.5)(3.5,3.5)(3.5,0.5)\drawline[1](1.5,2.5)(2.5,2.5)(2.5,0.5)
\drawline[1](0.5,2.5)(0.5,0.5)(1.5,0.5)
\put(0.5,0){\makebox(0,0){$1$}}\put(1,0){\makebox(0,0){$<$}}\put(1.5,0){\makebox(0,0){$2$}}
\put(2,0){\makebox(0,0){$<$}}\put(2.5,0){\makebox(0,0){$3$}}\put(3,0){\makebox(0,0){$<$}}
\put(3.5,0){\makebox(0,0){$4$}}
\end{picture}
\hspace{0.05\textwidth}
\begin{picture}(4,4)
\multiput(0.5,0.5)(0,1){4}{\multiput(0,0)(1,0){4}{\circle*{0.18}}}
\drawline[1](0.5,3.5)(3.5,3.5)(3.5,0.5)\drawline[1](1.5,2.5)(1.5,0.5)\drawline[1](2.5,2.5)(2.5,0.5)
\drawline[1](0.5,1.5)(1.5,1.5)\put(1.5,2.5){\oval(2,0.6)[t]}
\put(0.5,0){\makebox(0,0){$2$}}\put(1,0){\makebox(0,0){$\prec$}}\put(1.5,0){\makebox(0,0){$1$}}
\put(2,0){\makebox(0,0){$\sim$}}\put(2.5,0){\makebox(0,0){$3$}}\put(3,0){\makebox(0,0){$\prec$}}
\put(3.5,0){\makebox(0,0){$4$}}
\end{picture}
\hspace{0.06\textwidth}
\begin{picture}(6,4)
\put(1,0.5){\vector(1,0){5}}\put(1,0.5){\vector(0,1){3.2}}
\multiput(2,0.45)(1,0){4}{\line(0,1){0.1}}%
\multiput(0.95,1)(0,1){3}{\line(1,0){0.1}}%
\put(2,0){\makebox(0,0){$1$}}\put(3,0){\makebox(0,0){$2$}}\put(4,0){\makebox(0,0){$3$}}
\put(5,0){\makebox(0,0){$4$}}
\put(0.8,1){\makebox(0,0)[r]{$4$}}\put(0.8,2){\makebox(0,0)[r]{${1\sim 3}$}}\put(0.8,3){\makebox(0,0)[r]{$2$}}
\drawline[1](2,2)(3,3)(4,2)(5,1)
\put(2,2){\circle*{0.15}}\put(3,3){\circle*{0.15}}\put(4,2){\circle*{0.15}}\put(5,1){\circle*{0.15}}
\put(4,3){\makebox(0,0){\begin{normalsize}$f_{\precsim}$\end{normalsize}}}
\end{picture}
\end{small}
\caption{Example~\ref{ex:2134t}}
\label{fig:2134t}
\end{center}
\end{figure}

\begin{example}\label{ex:1423t}
Let us consider the operation $F\colon X_4^2\to X_4$ shown in Figure~\ref{fig:1423t} (left). Just as in Example~\ref{ex:2134t}, we can see that this operation is of the form \eqref{eq:kimura}, where $\precsim$ is the weak ordering on $X_4$ defined by $1\prec 4\prec 2\sim 3$; see Figure~\ref{fig:1423t} (center). Since $F$ is not $\leq_4$-preserving, by Proposition~\ref{prop:Qwspl} the weak ordering $\precsim$ is not weakly single-peaked for $\leq_4$. Here the graph of $f_{\precsim}$ is neither V-free, nor L-free, nor reversed L-free. It has the plateau $P=\{2,3\}$, which is not $\precsim$-minimal; see Figure~\ref{fig:1423t} (right).
\end{example}

\setlength{\unitlength}{4.5ex}
\begin{figure}[htbp]
\begin{center}
\begin{small}
\begin{picture}(4,4)
\multiput(0.5,0.5)(0,1){4}{\multiput(0,0)(1,0){4}{\circle*{0.18}}}
\drawline[1](1.5,3.5)(1.5,0.5)\drawline[1](2.5,3.5)(2.5,0.5)
\drawline[1](2.5,2.5)(3.5,2.5)\drawline[1](0.5,1.5)(1.5,1.5)
\put(1.5,2.5){\oval(2,0.6)[t]}\put(2.5,1.5){\oval(2,0.6)[t]}
\put(2,3.5){\oval(3,0.6)[t]}\put(3.5,2){\oval(0.6,3)[r]}
\put(0.5,0){\makebox(0,0){$1$}}\put(1,0){\makebox(0,0){$<$}}\put(1.5,0){\makebox(0,0){$2$}}
\put(2,0){\makebox(0,0){$<$}}\put(2.5,0){\makebox(0,0){$3$}}\put(3,0){\makebox(0,0){$<$}}
\put(3.5,0){\makebox(0,0){$4$}}
\end{picture}
\hspace{0.05\textwidth}
\begin{picture}(4,4)
\multiput(0.5,0.5)(0,1){4}{\multiput(0,0)(1,0){4}{\circle*{0.18}}}
\drawline[1](2.5,3.5)(2.5,0.5)\drawline[1](3.5,3.5)(3.5,0.5)\drawline[1](0.5,2.5)(2.5,2.5)
\drawline[1](0.5,1.5)(1.5,1.5)(1.5,0.5)\drawline[1](0.5,3.5)(1.5,3.5)\put(2.5,3.5){\oval(2,0.6)[t]}
\put(0.5,0){\makebox(0,0){$1$}}\put(1,0){\makebox(0,0){$\prec$}}\put(1.5,0){\makebox(0,0){$4$}}
\put(2,0){\makebox(0,0){$\prec$}}\put(2.5,0){\makebox(0,0){$2$}}\put(3,0){\makebox(0,0){$\sim$}}
\put(3.5,0){\makebox(0,0){$3$}}
\end{picture}
\hspace{0.06\textwidth}
\begin{picture}(6,4)
\put(1,0.5){\vector(1,0){5}}\put(1,0.5){\vector(0,1){3.2}}
\multiput(2,0.45)(1,0){4}{\line(0,1){0.1}}%
\multiput(0.95,1)(0,1){3}{\line(1,0){0.1}}%
\put(2,0){\makebox(0,0){$1$}}\put(3,0){\makebox(0,0){$2$}}\put(4,0){\makebox(0,0){$3$}}
\put(5,0){\makebox(0,0){$4$}}
\put(0.8,1){\makebox(0,0)[r]{${2\sim 3}$}}\put(0.8,2){\makebox(0,0)[r]{$4$}}\put(0.8,3){\makebox(0,0)[r]{$1$}}
\drawline[1](2,3)(3,1)(4,1)(5,2)
\put(2,3){\circle*{0.15}}\put(3,1){\circle*{0.15}}\put(4,1){\circle*{0.15}}\put(5,2){\circle*{0.15}}
\put(4,3){\makebox(0,0){\begin{normalsize}$f_{\precsim}$\end{normalsize}}}
\end{picture}
\end{small}
\caption{Example~\ref{ex:1423t}}
\label{fig:1423t}
\end{center}
\end{figure}

\begin{remark}
For any integer $n\geq 1$, the weak orderings $\precsim$ on $X=X_n$ that satisfy conditions (a) and (b') of Propositions~\ref{prop:52z} and \ref{prop:53z} are known in social choice theory as being \emph{single-plateaued for $\leq_n$} (see, e.g., \cite[Definition 4 and Lemma 17]{Fitz15}).\footnote{Both concepts of single-peakedness and single-plateauedness were introduced on finite domains by Black~\cite{Bla87}.} Thus, by Theorem~\ref{thm:VLL5} the weak orderings $\precsim$ on $X_n$ that are weakly single-peaked for $\leq_n$ are also single-plateaued for $\leq_n$ and vice versa. Since the graphical representations of these weak orderings need not include plateaus, we will keep our terminology and say that they are weakly single-peaked for $\leq_n$.
\end{remark}

We can now extend Proposition~\ref{prop:spConv} to weak orderings.

\begin{proposition}
Let $\leq$ be a total ordering on $X$ and let $\precsim$ be a weak ordering on $X$. Then condition (a) of Proposition~\ref{prop:52z} holds if and only if for every $t\in X$ the set $\{x\in X: x\precsim t\}$ is convex for $\leq$.
\end{proposition}

\begin{proof}
(Necessity) Let $t\in X$ and let $a,b,c\in X$ such that $a,c\in\{x\in X: x\precsim t\}$ and $b\in\cleq(a,c)$. By condition (a), we have $b\in\{x\in X: x\precsim t\}$.

(Sufficiency) For the sake of a contradiction, suppose that there exist $a,b,c\in X$ such that $b\in\cleq(a,c)$ and $\max_{\precsim}(a,c)\prec b$. Set $t_0=c$ if $a\prec c$, and $t_0=a$, otherwise. We then have $a,c\in\{x\in X: x\precsim t_0\}$. By convexity for $\leq$ we also have $b\in\{x\in X: x\precsim t_0\}$. Therefore we have $\max_{\precsim}(a,c)\prec b\precsim t_0$, which contradicts the definition of $t_0$.
\end{proof}

\begin{remark}
The dual version of weak single-peakedness can be defined exactly as we did for single-peakedness (see Remark~\ref{rem:dualsp}): just replace the condition $b\prec a$ or $b\prec c$ or $a\sim b\sim c$ by $a\prec b$ or $c\prec b$ or $a\sim b\sim c$. Here again, considering the weak single-peakedness property or its dual version is simply a matter of convention.
\end{remark}

\section{Conclusion}

This paper is rooted in a known characterization of associative and quasitrivial binary operations on an arbitrary set $X$, which essentially states that each of these operations can be thought of as a maximum with respect to a weak ordering (Theorem~\ref{thm:kimura}). We established different characterizations of the subclass of associative, quasitrivial, and commutative operations (Theorem~\ref{thm:char1}) and different characterizations of the subclass of associative, quasitrivial, commutative, and $\leq$-preserving operations when $X$ is endowed with a total ordering $\leq$ (Theorem~\ref{thm:cchar1}). When commutativity is no longer assumed, finding generalizations of these characterizations remains an interesting open question (see below).

When $X$ is an $n$-element set we also enumerated
\begin{itemize}
\item all associative and quasitrivial operations with or without neutral and/or annihilator elements (Theorem~\ref{thm:Nk} and Proposition~\ref{prop:qe}), thus solving an enumeration problem posed in \cite{CouDevMar},
\item all associative, quasitrivial, and $\leq$-preserving operations (when $X$ is endowed with a total ordering $\leq$) with or without neutral and/or annihilator elements (Propositions~\ref{prop:rn} and \ref{prop:ren}).
\end{itemize}

In order to characterize those that are $\leq$-preserving, we made use of single-peakedness. We proposed a generalization of this concept by introducing weak single-peakedness (Definition~\ref{de:wBlack} and Theorem~\ref{thm:QtAsNdKi}) and we provided a graphical characterization of the latter (Theorem~\ref{thm:VLL5}). When $X$ is an $n$-element set, we also enumerated all weak orderings on $X$ that are weakly single-peaked for the reference ordering on $X$ (Propositions~\ref{prop:vn} and \ref{prop:ven}). We posted in the Sloane's OEIS \cite{Slo} all the new sequences that arose from our results.

In view of these results, some questions emerge naturally and we now list a few below.

\begin{itemize}
\item Generalize Theorems~\ref{thm:char1} and \ref{thm:cchar1} by removing commutativity in assertion (i).
\item Analyze the asymptotic behavior of the sequence $(q(n))_{n\geq 0}$ (see Remark~\ref{rem:r4zt}(a)).
\item The integer sequences A000129, A002605, A048739, and A163271 were previously introduced in the OEIS to solve enumeration problems not related to weak single-peakedness and quasitrivial semigroups. It would be interesting to establish one-to-one correspondences between those problems and ours.
\item Find the number of operations $F\colon X_n^2\to X_n$ that are associative, quasitrivial, and $\leq$-preserving for some total ordering $\leq$ on $X$. The values for $1\leq n\leq 4$ are 1, 4, 20, 130. For instance the operation on $X_4$ represented in Figure~\ref{fig:nop} is associative and quasitrivial. However, there is no total ordering $\leq$ on $X_4$ for which this operation is $\leq$-preserving.
\end{itemize}

\section*{Acknowledgments}

The authors thank the reviewer for his/her quick and careful review and useful suggestions. They also thank Meltem \"Ozt\"urk of Paris-Dauphine University for bringing reference \cite{Fitz15} to their attention. This research is partly supported by the Internal Research Project R-AGR-0500 of the University of Luxembourg and by the Luxembourg National Research Fund R-AGR-3080.

\appendix
\section*{Appendix: alternative proof of formula~\eqref{eq:Nk}}

We provide an alternative proof of formula~\eqref{eq:Nk} that does not make use of the EGF of the sequence $(q(n))_{n\geq 0}$.

For any integer $n\geq 0$, define $s(n)=\min\{n+1,2\}$. Also, for any integer $k\geq 1$, let $\mathcal{P}_k$ be the vector space of real polynomial functions of $k$ variables, and let $T_k\colon\mathcal{P}_k\to\R$ be the linear transformation defined as
$$
T_k(P) ~=~ \sum_{I\subseteq\{1,\ldots,k\}}(-1)^{|I|+k}{\,}2^{|I|}\int_{[0,1]^k}P(t_1,\ldots,t_k)\big|_{t_i=0{\,}\forall i\notin I}~dt_1\cdots{\,}dt_k.
$$
For instance, if $P(x_1,\ldots,x_k)=\prod_{i=1}^k x_i^{m_i}$ for some integers $m_1,\ldots,m_k\geq 0$, then we have
\begin{eqnarray}
T_k(P) &=& \sum_{I\subseteq\{1,\ldots,k\}{\,}\text{s.t.}{\,}m_i=0{\,}\forall i\notin I}(-1)^{|I|+k}{\,}2^{|I|}{\,}\prod_{i\in I}\frac{1}{m_i+1}\nonumber\\
&=& \sum_{I\subseteq\{1,\ldots,k\}}~\bigg(\prod_{i\in I}\frac{2}{m_i+1}\bigg){\,}\bigg((-1)^{k-|I|}\prod_{i\in\{1,\ldots,k\}\setminus I}(2-s(m_i))\bigg)\nonumber\\
&=& \prod_{i=1}^k\Big(\frac{2}{m_i+1}-2+s(m_i)\Big) ~=~ \prod_{i=1}^k \frac{s(m_i)}{m_i+1}{\,},\label{eq:TjP}
\end{eqnarray}
where we have used the multi-binomial theorem
$$
\prod_{i=1}^k(x_i+y_i) ~=~ \sum_{I\subseteq\{1,\ldots,k\}}\prod_{i\in I}x_i\prod_{i\in\{1,\ldots,k\}\setminus I}y_i.
$$
Similarly, if $P(x_1,\ldots,x_k)=(\sum_{i=1}^kx_i)^{n-k}$ for some integer $n\geq k$, then we have
\begin{eqnarray}
T_k(P) &=& \sum_{I\subseteq\{1,\ldots,k\}}(-1)^{|I|+k}{\,}2^{|I|}\int_{[0,1]^k}\Big(\sum_{i\in I}t_i\Big)^{n-k}~dt_1\cdots{\,}dt_j\nonumber\\
&=& \sum_{i=0}^k{\,}(-1)^{i+k}{\,}{k\choose i}{\,}2^i{n-k+i\brace i}{n-k+i\choose i}^{-1},\label{eq:TjP2}
\end{eqnarray}
where we have used the formula (see, e.g., \cite[p.~202]{Jor79})
$$
\int_{[0,1]^n}\Big(\sum_{i=1}^nt_i\Big)^kdt_1\cdots{\,}dt_n ~=~ {k+n\brace n}{k+n\choose n}^{-1}\qquad\text{($k\geq 0$, $n\geq 1$, integers)}.
$$

We now use the results above to establish the claimed expression of $q(n)$.

Let us assume that $n\geq 1$. Using Theorem~\ref{thm:kimura} we can easily see that (see justification in the proof of Theorem~\ref{thm:Nk})
$$
q(n) ~=~ \sum_{k=1}^n ~\sum_{\textstyle{n_1+\cdots +n_k=n\atop n_1,\ldots,n_k\geq 1}}{n\choose n_1,\ldots,n_k}\prod_{\textstyle{i=1\atop n_i\geq 2}}^k 2{\,}.
$$
Setting $m_i=n_i-1$ for $i=1,\ldots,k$ in the latter formula, we obtain
\begin{eqnarray*}
q(n) &=& \sum_{k=1}^n ~\sum_{\textstyle{m_1+\cdots +m_k=n-k\atop m_1,\ldots,m_k\geq 0}}\frac{n!}{(m_1+1)!{\,}\cdots{\,}(m_k+1)!}{\,}\prod_{i=1}^k s(m_i)\\
&=& \sum_{k=1}^n\frac{n!}{(n-k)!} ~\sum_{\textstyle{m_1+\cdots +m_k=n-k\atop m_1,\ldots,m_k\geq 0}}\frac{(n-k)!}{m_1!{\,}\cdots{\,}m_k!}~\prod_{i=1}^k \frac{s(m_i)}{m_i+1}
\end{eqnarray*}
Using \eqref{eq:TjP}, the linearity of $T_k$, the multinomial theorem, and then \eqref{eq:TjP2}, we obtain
\begin{eqnarray*}
q(n) &=& \sum_{k=1}^n\frac{n!}{(n-k)!} ~\sum_{\textstyle{m_1+\cdots +m_k=n-k\atop m_1,\ldots,m_k\geq 0}}{n-k\choose m_1,\ldots,m_k}~T_k\Big(\prod_{i=1}^k x_i^{m_i}\Big)\\
&=& \sum_{k=1}^n\frac{n!}{(n-k)!} ~T_k\bigg(\sum_{\textstyle{m_1+\cdots +m_k=n-k\atop m_1,\ldots,m_k\geq 0}}{n-k\choose m_1,\ldots,m_k}{\,}\prod_{i=1}^k x_i^{m_i}\bigg)\\
&=& \sum_{k=1}^n\frac{n!}{(n-k)!} ~T_k\bigg(\Big(\sum_{i=1}^kx_i\Big)^{n-k}\bigg)\\
&=& \sum_{k=1}^n k!{\,}\sum_{i=0}^k(-1)^{i+k}{\,}2^i{n\choose k-i}{n-k+i\brace i}.
\end{eqnarray*}

Permuting the sums in the last expression and observing that ${n\brace 0}=0$, we finally obtain
$$
q(n) ~=~ \sum_{i=0}^n(-2)^i{\,}\sum_{k=i}^n(-1)^k{\,}{n\choose k-i}{n-k+i\brace i}k!{\,},
$$
from which we immediately derive the claimed expression of $q(n)$.\qed

\end{document}